\documentclass[a4paper]{amsart} 
\usepackage{amssymb,times, amscd,amsmath,amsthm, xypic}
\usepackage[all]{xy}
\usepackage{geometry}

\author{Shoji Yokura$^{(*)}$}

\address
{Department of Mathematics and Computer Science, 
Faculty of Science, 
Kagoshima University, 21-35 Korimoto 1-chome, Kagoshima 890-0065, Japan}
\email {yokura@sci.kagoshima-u.ac.jp}
\title [Oriented bivariant theory II]
{Oriented bivariant theory II\\
-- Algebraic cobordism of $S$-schemes --}

\thanks {(*) Partially supported by JSPS KAKENHI Grant Number 16H03936 \\
\quad \emph{keywords} : (co)bordism, algebraic cobordism \\
\quad \emph{Mathematics Subject Classification 2000}: 55N35, 55N22, 14C17, 14C40, 14F99, 19E99}

\begin{document} 
\numberwithin{equation}{section}
\newtheorem{thm}[equation]{Theorem}
\newtheorem{pro}[equation]{Proposition}
\newtheorem{prob}[equation]{Problem}
\newtheorem{cor}[equation]{Corollary}
\newtheorem{lem}[equation]{Lemma}
\theoremstyle{definition}
\newtheorem{ex}[equation]{Example}
\newtheorem{defn}[equation]{Definition}
\newtheorem{rem}[equation]{Remark}
\renewcommand{\rmdefault}{ptm}
\def\alp{\alpha}
\def\be{\beta}
\def\jeden{1\hskip-3.5pt1}
\def\om{\omega}
\def\bigstar{\mathbf{\star}}
\def\ep{\epsilon}
\def\vep{\varepsilon}
\def\Om{\Omega}
\def\la{\lambda}
\def\La{\Lambda}
\def\si{\sigma}
\def\Si{\Sigma}
\def\Cal{\mathcal}
\def\ga{\gamma}
\def\Ga{\Gamma}
\def\de{\delta}
\def\De{\Delta}
\def\bF{\mathbb{F}}
\def\bH{\mathbb H}
\def\bPH{\mathbb {PH}}
\def \bB{\mathbb B}
\def \bA{\mathbb A}
\def \bOB{\mathbb {OB}}
\def \bM{\mathbb M}
\def \bOM{\mathbb {OM}}
\def \calB{\mathcal B}
\def \bK{\mathbb K}
\def \bG{\mathbf G}
\def \bL{\mathbb L}
\def\bN{\mathbb N}
\def\bR{\mathbb R}
\def\bP{\mathbb P}
\def\bZ{\mathbb Z}
\def\bC{\mathbb  C}
\def \bQ{\mathbb Q}
\def\op{\operatorname}

\maketitle

\begin{abstract} This is a sequel to our previous paper of oriented bivariant theory \cite{Yokura-obt}. In 2001 M. Levine and F. Morel constructed algebraic cobordism $\Omega_*(X)$ for schemes $X$ over a field $k$ in an abstract way and later M. Levine and  R. Pandhairpande reconstructed it more geometrically. In this paper in a similar manner we construct an algebraic cobordism $\Omega^*(X \xrightarrow {\pi_X} S)$ for a scheme $X$ over a fixed scheme $S$ in such a way that if the target scheme $S$ is the point $pt = \op{Spec} k$, then $\Omega^{-i}(X \xrightarrow {\pi_X} pt)$ is isomorphic to Levine--Morel's algebraic cobordism $\Omega_i(X)$.\end{abstract}


\section{Introduction}\label{intro} 

V. Voevodsky has introduced algebraic cobordism (now called {\it higher algebraic cobordism}), which was used in his proof of Milnor's conjecture \cite{Voevodsky}. D. Quillen introduced the notion of {\it (complex ) oriented cohomology theory} on the category of differential manifolds \cite{Quillen} and this notion can be formally extended to the category of smooth schemes in algebraic geometry. M. Levine and F. Morel constructed {\it a universal oriented cohomology theory} on smooth schemes, which  they also call {\it algebraic cobordism} \cite{LM},  and recently  M. Levine and R. Pandharipande \cite{LP} gave another equivalent construction of the algebraic cobordism via what they call ``double point degeneration" and they found a nice application of it in the Donaldson--Thomas theory of 3-folds, i.e. they proved what is called MNOP conjecture \cite{MNOP}. 

The algebraic cobordism $\Omega_*(X)$ of a scheme $X$ is roughly speaking constructed as follows. First they consider the Grothendieck group or the group completion, denoted by $\Cal Z_*(X)$, of the monoid consisting of isomorphism classes $[M \xrightarrow {h} X, L_1, \cdots, L_r]$ of a projective morphism $h:M \to X$ from a quasi-projective smooth scheme $M$ together with line bundles $L_i$ over the source scheme $M$. The functor $\Cal Z_*$ carries the four data (D1), (D2), (D3), (D4), and they satisfy the eight conditions (A1), (A2), $\cdots$, (A8). These data and conditions are not written here, but in the following section we will write them in our context. A functor having such four data and satisfying the eight conditions is called \emph{an oriented Borel--Moore functor with products}.

If $A_*$ is an oriented Borel--Moore functor with products, then the abelian group $A_*(pt)$ of the point $pt$ becomes a commutative graded ring. Given a commutative ring $R_*$, an oriented Borel-Moore functor with product $A_*$ together with a graded ring homomorphism $R_* \to A_*(pt)$ is called \emph{an oriented Borel--Moore $R_*$-functor with products}. 
Let $\mathbb L_*$ be the Lazard ring. Then, \emph{an oriented Borel--Moore functor with products {\bf of geometric type}} is defined to be an oriented Borel--Moore $\mathbb L_*$-functor with products which satisfies the following three axioms (see \cite[Definition 2.2.1]{LM}):
\begin{itemize}
\item (Dim) Dimension Axiom: For any smooth scheme $Y$ and any family $(L_1,\cdots, L_n)$ of line bundles on $Y$ with $n > \op{dim}Y$, one has
$$[Y \xrightarrow {\op{id}_Y} Y; L_1, \cdots, L_n] = 0 \in A_*(Y).$$
\item (Sect) Section Axiom: For any smooth scheme $Y$, any line bundles $L$ on $Y$ and any section $s$ of $L$ which is transverse to the zero section of $L$, one has
$$[Y \xrightarrow {\op{id}_Y} Y; L] = i_*[Z \xrightarrow {\op{id}_Z} Z],$$
where $i:Z \to Y$ is the closed immersion.
\item (FGL) Formal Group Law Axiom:  Let $\phi_A:\mathbb L_* \to A_*(pt)$ be the ring homomorphism giving the $\mathbb L_*$-structure and let $F_A \in A_*(pt)[[u,v]]$ be the image of the universal formal group law $F_{\mathbb L} \in \mathbb L_*[[u,v]]$ by $\phi_A$. Then for any smooth scheme $Y$ and any pair $(L,M)$ of line bundles on $Y$, one has
$$F_A(\widetilde c_1(L), \widetilde c_1(M))([Y \xrightarrow {\op{id}_Y} Y]) = \widetilde c_1(L \otimes M)([Y \xrightarrow {\op{id}_Y} Y]) \in A_*(Y).$$
\end{itemize}
\emph{The above group $\Cal Z_*(X)$ modded out by a certain subgroup $\Cal R_*(X)$ involving the above three axioms,
$$\Omega_*(X) := \frac{\Cal Z_*(X)}{\Cal R_*(X)}$$
becomes an oriented Borel--Moore functor with products of geometric type.} 
To be a bit more precisely, they construct it step by step. 
\begin{enumerate}
\item First they consider the subgroup $\langle R^{Dim}\rangle(X) \subset \mathcal Z_*(X)$ dealing with (Dim) 
and define the quotient 
$$\underline{\mathcal Z_*}(X) := \frac{\mathcal Z_*(X) }{\langle R^{Dim}\rangle(X)}.$$
\item Secondly, they consider the subgroup $\langle R^{Sect}\rangle(X) \subset \underline{\mathcal Z_*}(X)$ 
dealing with (Sect) 
and define the quotient 
$$\underline{\underline{\mathcal Z_*}}(X) := \frac{\underline {\mathcal Z_*}(X) }{\langle R^{Sect}\rangle(X)}.$$
\item Finally, they consider the subgroup $\langle R^{FGL}\rangle(X) \subset \mathbb L_* \otimes \underline{\underline{\mathcal Z_*}}(X)$ dealing with (FGL) 
and define the quotient 
$$\Omega_*(X) := \frac{\mathbb L_* \otimes \underline{\underline{\mathcal Z_*}}(X)}{\langle R^{FGL}\rangle(X)}.$$
\end{enumerate}
It turns out (via the construction of $\Omega_*(X)$) that Levine--Morel's $\Omega_*(X)$ is the universal one among such oriented Borel--Moore functors with products of geometric type. The main theorem of \cite{LM} is that if the ground field is of characteristic zero (because the resolution of singularities is used) and if  we restrict ourselves to the category of smooth schemes $X$ the theory of algebraic cobordism $\Omega^*(X):= \Omega_{\op {dim} X - *}(X)$ is in fact \emph{the universal oriented cohomology theory}\footnote{In this sense Levine-Morel's algebraic cobordism $\Omega_*(X)$ is a bordism theory, thus could be called ``algebraic bordism".}. The group $\Cal Z_*(X)$ shall be called the \emph{pre-algebraic cobordism} of $X$.

In the definition of the algebraic cobordism \cite{LM} (also see \cite{LP}), they consider projective morphisms from \emph{quasi-projective} smooth varieties, or equivalently proper morphisms from \emph{quasi-projective} smooth varieties. Very recently J. L. Gonz\'alez and K. Karu \cite{GK2} (cf. \cite{GK})  have observed that the assumption of quasi-projectivity can be dropped, i.e., one can consider proper morphisms from smooth varieties, to get the same algebraic cobordism. So in the rest of the paper we will use Gonz\'alez-Karu's description. 

In our previous paper \cite{Yokura-obt}\footnote{Using the same idea as in \cite{Yokura-obt}, in \cite{SY1} we construct a bivariant version of what is called a motivic characteristic class \cite{BSY}.}, aiming at the construction of a bivariant version of Levine-Morel's algebraic cobordism \cite{LM} (see also \cite{LP}), we introduce an oriented bivariant theory. The starting point of this project was that we observed that for any bivariant theory $\bB$, the covariant theory $\bB_*(X) := \bB^{-*}(X \to pt)$ and the contravariant theory $\bB^*(X) := \bB^*(X \xrightarrow{\op{id}_X} X)$ are both more or less what is called \emph{a Borel--Moore functor with products} without the Chern operator defined (if the Chern operator is also defined and compatible with the pushforward, pullback and exterior product, it is called \emph{an oriented Borel--Moore functor with products}). 

In this paper we start with our following oriented bivariant theory 
$$\mathcal Z^*(X \xrightarrow f Y) := \Bigl \{ [V \xrightarrow h X; L_1, \cdots, L_r] \Bigr \}$$
which is the graded abelian group generated by the coboridsm cycle $[V \xrightarrow h X; L_1, \cdots, L_r]$ (as defined in \cite{LM}) such that
\begin{enumerate}
\item $h:V \to X$ is proper,
\item the composite $f \circ h:V \to Y$ is smooth,
\item $L_i$ is a line bundle over $V$.
\end{enumerate}
The grading is defined by
$$ [V \xrightarrow h X; L_1, \cdots, L_r] \in \mathcal Z^i(X \xrightarrow f Y) \Longleftrightarrow -i+r = \op{dim}(f \circ h),$$
where the dimension $\op{dim}(f \circ h)$ is the relative dimension of the smooth morphism, i.e., the dimension of the fiber, or $\op{dim}(f \circ h) =\op{dim} V - \op{dim}Y$. 
\begin{rem}
\begin{enumerate}
\item The reason why we consider such gradings is that eventually we want to capture $\Omega^*(X \xrightarrow {\pi_X} S)$ as a bivariant-theoretical group in the sense of Fulton--MacPherson \cite{FM}. In their bivariant theory $\mathbb B$, the grading is such that $\mathbb B^{-i}(X \to pt) = \mathbb B_i(X)$, which is the associated covariant group.
\item When $Y$ is a point, then our $\mathcal Z^*(X \xrightarrow f pt)$ is the same as Levine-Morel's pre-algebraic cobordism $\mathcal Z_*(X)$ .
\end{enumerate}
\end{rem}

In this paper we restrict the above group to the category of $S$-schemes over a fixed scheme, namely the over category $Sch/S$, whose objects are morphisms $\pi_X:X \to S$ and morphisms from $\pi_X:X \to S$ to $\pi_Y:Y \to S$ are $f:X \to Y$ 
such that the following diagram commutes
$$\xymatrix{
X\ar[dr]_ {\pi_X}\ar[rr]^ {f} && Y \ \ar[dl]^{\pi_Y}\\
&S.}$$
Thus we consider the graded abelian group $\mathcal Z^*(X \xrightarrow {\pi_X} S)$ on the over category $Sch/S$.
Then, in an analogous manner as done in Levine-Morel's construction, we proceed as follows:
\begin{enumerate}
\item First, consider the subgroup $\langle R^{Dim}\rangle(X \xrightarrow {\pi_X} S) \subset \mathcal Z_*(X \xrightarrow {\pi_X} S)$ dealing with (rel-Dim) a ``relative" Dimension Axiom
and define the quotient 
$$\underline{\mathcal Z^*}(X \xrightarrow {\pi_X} S) := \frac{\mathcal Z^*(X \xrightarrow {\pi_X} S) }{\langle R^{Dim}\rangle(X \xrightarrow {\pi_X} S)}.$$
\item Secondly, consider the subgroup $\langle R^{Sect}\rangle(X \xrightarrow {\pi_X} S) \subset \underline{\mathcal Z^*}(X \xrightarrow {\pi_X} S)$ 
dealing with (rel-Sect) a ``relative" Section Axiom, 
and define the quotient 
$$\underline{\underline{\mathcal Z^*}}(X \xrightarrow {\pi_X} S) := \frac{\underline {\mathcal Z^*}(X \xrightarrow {\pi_X} S) }{\langle R^{Sect}\rangle(X \xrightarrow {\pi_X} S)}.$$
\item Finally, consider the subgroup $\langle R^{FGL}\rangle(X \xrightarrow {\pi_X} S) \subset \mathbb L_* \otimes \underline{\underline{\mathcal Z^*}}(X \xrightarrow {\pi_X} S)$ dealing with (FGL) a ``relative" Formal Group Law Axiom, 
and define the quotient 
$$\Omega^*(X \xrightarrow {\pi_X} S) := \frac{\mathbb L_* \otimes \underline{\underline{\mathcal Z^*}}(X \xrightarrow {\pi_X} S)}{\langle R^{FGL}\rangle(X \xrightarrow {\pi_X} S)}.$$
\end{enumerate}

$\Omega^*(X \xrightarrow {\pi_X} S)$ is a $Sch/S$-version of an oriented Borel--Moore functor with products of geometric type for a scheme $X$ over a fixed scheme $S$ such that $\Omega^{-*}(X \to pt)$ is equal to Levine--Morel's algebraic cobordism $\Omega_*(X)$. In this sense our group $\Omega^*(X \xrightarrow {\pi_X} S)$, algebraic cobordism of an $S$-scheme $\pi_X:X \to S$, could be called a \emph{relative algebraic cobordism}.
 
\begin{rem}
Motivated by the construction given in the present paper, furthermore we tried to get a bivariant version of algebraic cobordism, which is our ultimate goal. However we have been unable to obtain one, mainly because the bivariant product is not well-defined under the present construction, as we will make a remark later. We hope that by modifying the present construction slightly we would be able to get a bivariant algebraic cobordism.
\end{rem}
\section {Fulton--MacPherson's bivariant theory and a universal bivariant theory}\label{FM-BT}

We make a quick review of Fulton--MacPherson's bivariant theory \cite {FM} (also see \cite{Fulton-book}) and a universal bivariant theory \cite{Yokura-obt}. 

Let $\Cal V$ be a category which has a final object $pt$ and on which the fiber product or fiber square is well-defined. Also we consider a class of maps, called ``confined maps" (e.g., proper maps, projective maps, in algebraic geometry), which are \emph{closed under composition and base change and contain all the identity maps}, and a class of fiber squares, called ``independent squares" (or ``confined squares", e.g., ``Tor-independent" in algebraic geometry, a fiber square with some extra conditions required on morphisms of the square), which satisfy the following:

(i) if the two inside squares in  
$$\CD
X''@> {h'} >> X' @> {g'} >> X \\
@VV {f''}V @VV {f'}V @VV {f}V\\
Y''@>> {h} > Y' @>> {g} > Y \endCD
\quad \quad \qquad \text{or} \qquad \quad \quad 
\CD
X' @>> {h''} > X \\
@V {f'}VV @VV {f}V\\
Y' @>> {h'} > Y \\
@V {g'}VV @VV {g}V \\
Z'  @>> {h} > Z \endCD
$$
are independent, then the outside square is also independent,

(ii) any square of the following forms are independent:
$$
\xymatrix{X \ar[d]_{f} \ar[r]^{\op {id}_X}&  X \ar[d]^f & & X \ar[d]_{\op {id}_X} \ar[r]^f & Y \ar[d]^{\op {id}_Y} \\
Y \ar[r]_{\op {id}_X}  & Y && X \ar[r]_f & Y}
$$
where $f:X \to Y$ is \emph{any} morphism. 

A bivariant theory $\bB$ on a category $\Cal V$ with values in the category of graded abelian groups is an assignment to each morphism
$ X  \xrightarrow{f} Y$
in the category $\Cal V$ a graded abelian group (in most cases we ignore the grading )
$\bB(X  \xrightarrow{f} Y)$
which is equipped with the following three basic operations. The $i$-th component of $\bB(X  \xrightarrow{f} Y)$, $i \in \bZ$, is denoted by $\bB^i(X  \xrightarrow{f} Y)$.
\begin{enumerate}
\item {\bf Product}: For morphisms $f: X \to Y$ and $g: Y
\to Z$, the product operation
$$\bullet: \bB^i( X  \xrightarrow{f}  Y) \otimes \bB^j( Y  \xrightarrow{g}  Z) \to
\bB^{i+j}( X  \xrightarrow{gf}  Z)$$
is  defined.

\item {\bf Pushforward}: For morphisms $f: X \to Y$
and $g: Y \to Z$ with $f$ \emph {confined}, the pushforward operation
$$f_*: \bB^i( X  \xrightarrow{gf} Z) \to \bB^i( Y  \xrightarrow{g}  Z) $$
is  defined.

\item {\bf Pullback} : For an \emph{independent} square \qquad $\CD
X' @> g' >> X \\
@V f' VV @VV f V\\
Y' @>> g > Y, \endCD
$

the pullback operation
$$g^* : \bB^i( X  \xrightarrow{f} Y) \to \bB^i( X'  \xrightarrow{f'} Y') $$
is  defined.
\end{enumerate}

These three operations are required to satisfy the following seven compatibility axioms (\cite [Part I, \S 2.2]{FM}):
\begin{enumerate}

\item[($A_1$)] {\bf Product is associative}: given a diagram $X \xrightarrow f Y  \xrightarrow g Z \xrightarrow h  W$ with $\alp \in \bB(X \xrightarrow f Y),  \be \in \bB(Y \xrightarrow g Z), \ga \in \bB(Z \xrightarrow h W)$,
$$(\alp \bullet\be) \bullet \ga = \alp \bullet (\be \bullet \ga).$$
\item[($A_2$)] {\bf Pushforward is functorial} : given a diagram $X \xrightarrow f Y  \xrightarrow g Z \xrightarrow h  W$ with $f$ and $g$ confined and $\alp \in \bB(X \xrightarrow {h\circ g\circ f} W)$
$$(g\circ f)_* (\alp) = g_*(f_*(\alp)).$$
\item[($A_3$)] {\bf Pullback is functorial}: given independent squares
$$\CD
X''@> {h'} >> X' @> {g'} >> X \\
@VV {f''}V @VV {f'}V @VV {f}V\\
Y''@>> {h} > Y' @>> {g} > Y \endCD
$$
and $\alp \in \bB(X \xrightarrow f Y)$,
$$(g \circ h)^*(\alp) = h^*(g^*(\alp)).$$

\item[($A_{12}$)] {\bf Product and pushforward commute}: given a diagram $X \xrightarrow f Y  \xrightarrow g Z \xrightarrow h  W$ with $f$ confined and $\alp \in \bB(X \xrightarrow {g \circ f} Z),  \be \in \bB(Z \xrightarrow h W)$,
$$f_*(\alp \bullet\be)  = f_*(\alp) \bullet \be.$$
\item[($A_{13}$)] {\bf Product and pullback commute}: given independent squares
$$\CD
X' @> {h''} >> X \\
@V {f'}VV @VV {f}V\\
Y' @> {h'} >> Y \\
@V {g'}VV @VV {g}V \\
Z'  @>> {h} > Z \endCD
$$
with $\alp \in \bB(X \xrightarrow {f} Y),  \be \in \bB(Y \xrightarrow g Z)$,
$$h^*(\alp \bullet\be)  = {h'}^*(\alp) \bullet h^*(\be).$$
\item[($A_{23}$)] {\bf Pushforward and pullback commute}: given independent squares
$$\CD
X' @> {h''} >> X \\
@V {f'}VV @VV {f}V\\
Y' @> {h'} >> Y \\
@V {g'}VV @VV {g}V \\
Z'  @>> {h} > Z \endCD
$$
with $f$ confined and $\alp \in \bB(X \xrightarrow {g\circ f} Z)$,
$$f'_*(h^*(\alp))  = h^*(f_*(\alp)).$$
\item[($A_{123}$)] {\bf Projection formula}: given an independent square with $g$ confined and $\alp \in \bB(X \xrightarrow {f} Y),  \be \in \bB(Y' \xrightarrow {h \circ g} Z)$
$$\CD
X' @> {g'} >> X \\
@V {f'}VV @VV {f}V\\
Y' @>> {g} > Y @>> h >Z \\
\endCD
$$
and $\alp \in \bB(X \xrightarrow {f} Y),  \be \in \bB(Y' \xrightarrow {h \circ g} Z)$,
$$g'_*(g^*(\alp) \bullet \be)  = \alp \bullet g_*(\be).$$
\end{enumerate}
We also assume that $\bB$ has units:

\underline {Units}: $\bB$ has units, i.e., there is an element $1_X \in \bB^0( X  \xrightarrow{\op {id}_X} X)$ such that $\alp \bullet 1_X = \alp$ for all morphisms $W \to X$ and all $\alp \in \bB(W \to X)$, such that $1_X \bullet \beta = \beta $ for all morphisms $X \to Y$ and all $\beta \in \bB(X \to Y)$, and such that $g^*1_X = 1_{X'}$ for all $g: X' \to X$. 

\underline {Commutativity}\label{commutativity}: $\bB$ is called \emph{commutative} if whenever both
$$\CD
W @> {g'} >> X \\
@V {f'}VV @VV {f}V\\
Y @>> {g} > Z  \\
\endCD  
\quad \quad \text{and} \quad \quad 
\CD
W @> {f'} >> Y \\
@V {g'}VV @VV {g}V\\
X @>> {g} > Z \\
\endCD  
$$
are independent squares with $\alp \in \bB(X \xrightarrow f Z)$ and $\be \in \bB(Y \xrightarrow g Z)$,
$$g^*(\alp) \bullet \be = f^*(\be) \bullet \alp .$$
(NOTE:If $g^*(\alp) \bullet \be = (-1)^{\op{deg}(\alp) \op{deg}(\be)}f^*(\be) \bullet \alp$ holds, it is called \emph{skew-commutative}. In this paper we assume that bivariant theories are commutative.)
Let $\bB, \bB'$ be two bivariant theories on a category $\Cal V$. A {\it Grothendieck transformation} from $\bB$ to $\bB'$, $\ga : \bB \to \bB'$
is a collection of homomorphisms
$\bB(X \to Y) \to \bB'(X \to Y)$
for a morphism $X \to Y$ in the category $\Cal V$, which preserves the above three basic operations: 
\begin{enumerate}
\item $\ga (\alp \bullet_{\bB} \be) = \ga (\alp) \bullet _{\bB'} \ga (\be)$, 
\item $\ga(f_{*}\alp) = f_*\ga (\alp)$, and 
\item $\ga (g^* \alp) = g^* \ga (\alp)$. 
\end{enumerate}
A bivariant theory unifies both a covariant theory and a contravariant theory in the following sense:

$\bB_*(X):= \bB(X \to pt)$ becomes a covariant functor for {\it confined}  morphisms and 

$\bB^*(X) := \bB(X  \xrightarrow{id}  X)$ becomes a contravariant functor for {\it any} morphisms. 
\noindent
A Grothendieck transformation $\ga: \bB \to \bB'$ induces natural transformations $\ga_*: \bB_* \to \bB_*'$ and $\ga^*: \bB^* \to {\bB'}^*$.

\begin{defn}\label{grading}
As to the grading, $\bB_i(X):= \bB^{-i}(X  \xrightarrow{id}  X)$ and
$\bB^j(X):= \bB^j(X  \xrightarrow{id}  X)$.
\end{defn}
 
\begin{defn}\label{canonical}(\cite[Part I, \S 2.6.2 Definition]{FM}) Let $\Cal S$ be a class of maps in $\Cal V$, which is closed under compositions and containing all identity maps. Suppose that to each $f: X \to Y$ in $\Cal S$ there is assigned an element
$\theta(f) \in \bB(X  \xrightarrow {f} Y)$ satisfying that
\begin{enumerate}
\item [(i)] $\theta (g \circ f) = \theta(f) \bullet \theta(g)$ for all $f:X \to Y$, $g: Y \to Z \in \Cal S$ and
\item [(ii)] $\theta(\op {id}_X) = 1_X $ for all $X$ with $1_X \in \bB^*(X):= B(X  \xrightarrow{\op {id}_X} X)$ the unit element.
\end{enumerate}
Then $\theta(f)$ is called an {\it orientation} of $f$. (In \cite[Part I, \S 2.6.2 Definition]{FM} it is called a {\it canonical orientation} of $f$, but in this paper it shall be simply called an orientation.)
\end{defn} 
\noindent
\underline {Gysin homomorphisms}\label{gysin}:
Note that such an orientation makes the covariant functor $\bB_*(X)$ a contravariant functor for morphisms in $\Cal S$, and also makes the contravariant functor $\bB^*$ a covariant functor for morphisms in $\Cal C \cap \Cal S$: Indeed, 
\begin{enumerate}
\item As to the covariant functor $\bB_*(X)$: For a morphism $f: X \to Y \in \Cal S$ and the orientation $\theta$ on $\Cal S$ the following {\it Gysin homomorphism}
$$f^*: \bB_*(Y) \to \bB_*(X) \quad \text {defined by} \quad  f^*(\alp) :=\theta(f) \bullet \alp$$
 is {\it contravariantly functorial}. 
\item As to contravariant functor $\bB^*$: For a fiber square (which is an independent square by hypothesis)
$$\CD
X @> f >> Y \\
@V {\op {id}_X} VV @VV {\op {id}_Y}V\\
X @>> f > Y, \endCD
$$
where $f \in \Cal C \cap  \Cal S$, the following {\it Gysin homomorphism}
$$f_*: \bB^*(X) \to \bB^*(Y) \quad \text {defined by} \quad
f_*(\alp) := f_*(\alp \bullet \theta (f))$$
is {\it covariantly functorial}.
\end{enumerate}

The notation should carry the information of $\Cal S$ and the orientation $\theta$, but it will be usually omitted if it is not necessary to be mentioned. Note that the above conditions (i) and (ii) of Definition (\ref{canonical}) are certainly \emph{necessary} for the above Gysin homomorphisms to be functorial. 

\begin{defn} (i) Let $\Cal S$ be another class of maps called ``specialized maps" (e.g., smooth maps in algebraic geometry) in $\Cal V$ , which is closed under composition, closed under base change and containing all identity maps. Let $\bB$ be a bivariant theory. If $\Cal S$ has  orientations in $\bB$, then we say that $\Cal S$ is $\bB$-oriented and an element of $\Cal S$ is called a $\bB$-oriented morphism. (Of course $\Cal S$ is also a class of confined maps, but since we consider the above extra condition of $\bB$-orientation  on $\Cal S$, we give a different name to $\Cal S$.)

(ii) Let $\Cal S$ be as in (i). Let $\bB$ be a bivariant theory and $\Cal S$ be $\bB$-oriented. Furthermore, if the orientation $\theta$ on $\Cal S$ satisfies that for an independent square with $f \in \Cal S$
$$
\CD
X' @> g' >> X\\
@Vf'VV   @VV f V \\
Y' @>> g > Y
\endCD
$$
the following condition holds: 
$\theta (f') = g^* \theta (f)$, 
(which means that the orientation $\theta$ preserves the pullback operation), then we call $\theta$ a {\it stable orientation} and say that $\Cal S$ is {\it stably $\bB$-oriented} and an element of $\Cal S$ is called {\it a stably $\bB$-oriented morphism} .
\end{defn}

The following theorem is about {\it the existence of a universal one} of the bivariant theories for a given category $\Cal V$ with a class $\Cal C$ of confined morphisms, a class of independent squares and a class $\Cal S$ of specialized morphisms. 

\begin{thm}(\cite[Theorem 3.1]{Yokura-obt})(A universal bivariant theory) \label{UBT} Let  $\Cal V$ be a category with a class $\Cal C$ of confined morphisms, a class of independent squares and a class $\Cal S$ of specialized maps.  We define 
$$\bM^{\Cal C} _{\Cal S}(X  \xrightarrow{f}  Y)$$
to be the free abelian group generated by the set of isomorphism classes of confined morphisms $h: W \to X$  such that the composite of  $h$ and $f$ is a specialized map:
$$h \in \Cal C \quad \text {and} \quad f \circ h: W \to Y \in \Cal S.$$
\begin{enumerate}
\item  The association $\bM^{\Cal C} _{\Cal S}$ is a bivariant theory if the three bivariant operations are defined as follows:
\begin{enumerate}
\item  {\bf Product}: For morphisms $f: X \to Y$ and $g: Y
\to Z$, the product operation
$$\bullet: \bM^{\Cal C} _{\Cal S} ( X  \xrightarrow{f}  Y) \otimes \bM^{\Cal C} _{\Cal S} ( Y  \xrightarrow{g}  Z) \to
\bM^{\Cal C} _{\Cal S} ( X  \xrightarrow{gf}  Z)$$
is  defined by
$$[V \xrightarrow{h}  X] \bullet [W  \xrightarrow{ k}  Y]:= [V'  \xrightarrow{ h \circ {k}''}  X]$$
and extended linearly, where we consider the following fiber squares
$$\CD
V' @> {h'} >> X' @> {f'} >> W \\
@V {{k}''}VV @V {{k}'}VV @V {k}VV\\
V@>> {h} > X @>> {f} > Y @>> {g} > Z .\endCD
$$
\item {\bf Pushforward}: For morphisms $f: X \to Y$
and $g: Y \to Z$ with $f$ \underline {confined}, the pushforward operation
$$f_*: \bM^{\Cal C} _{\Cal S} ( X  \xrightarrow{gf} Z) \to \bM^{\Cal C} _{\Cal S} ( Y  \xrightarrow{g}  Z) $$
is  defined by
$$f_*\left ([V \xrightarrow{h}  X] \right) := [V  \xrightarrow{f \circ h}  Y]$$
and extended linearly.

\item {\bf Pullback}: For an independent square
$$\CD
X' @> g' >> X \\
@V f' VV @VV f V\\
Y' @>> g > Y, \endCD
$$
the pullback operation
$$g^* : \bM^{\Cal C} _{\Cal S} ( X  \xrightarrow{f} Y) \to \bM^{\Cal C} _{\Cal S}( X'  \xrightarrow{f'} Y') $$
is  defined by
$$g^*\left ([V  \xrightarrow{h}  X] \right):=  [V'  \xrightarrow{{h}'}  X']$$
and extended linearly, where we consider the following fiber squares:
$$\CD
V' @> g'' >> V \\
@V {{h}'} VV @VV {h}V\\
X' @> g' >> X \\
@V f' VV @VV f V\\
Y' @>> g > Y. \endCD
$$
\end{enumerate}
\item Let $\Cal {BT}$ be a class of bivariant theories $\bB$ on the same category $\Cal V$ with a class $\Cal C$ of confined morphisms, a class of independent squares and a class $\Cal S$ of specialized maps. 
Let $\Cal S$ be \underline {nice} canonical $\bB$-orientable for any bivariant theory $\bB \in \Cal {BT}$. Then, for each bivariant theory $\bB \in \Cal {BT}$ there exists a unique Grothendieck transformation
$$\ga_{\bB} : \bM^{\Cal C} _{\Cal S} \to \bB$$
such that for a specialized morphism $f: X \to Y \in \Cal S$ the homomorphism
$\ga_{\bB} : \bM^{\Cal C} _{\Cal S}(X  \xrightarrow{f}  Y) \to \bB(X  \xrightarrow{f}  Y)$
satisfies the normalization condition that $$\ga_{\bB}([X  \xrightarrow{\op {id}_X}  X]) = \theta_{\bB}(f).$$
\end{enumerate}
\end{thm}

\section{Oriented bivariant theory and a universal oriented bivariant theory }

Levine--Morel's algebraic cobordism is the universal one among the so-caled  \emph {oriented} Borel--Moore functors with products for algebraic schemes. Here ``oriented" means that the given Borel--Moore functor $H_*$ is equipped with the endomorphsim $\tilde c_1(L): H_*(X) \to H_*(X)$ for a line bundle $L$ over the scheme $X$. Motivated by this ``orientation" (which is different from the one given in Definition \ref{canonical}, but we still call this ``orientation" using a different symbol so that the reader will not be confused with terminologies), in \cite[\S4]{Yokura-obt} we introduce an orientation to bivariant theories for any category, using the notion of \emph {fibered categories} in abstract category theory (e.g, see \cite{Vistoli}) and such a bivariant theory equipped with such an orientation (Chern operator) is called \emph{an oriented bivariant theory}. 

\begin{defn}
Let $\Cal L$ be a fibered category over $\Cal V$. An object in the fiber $\Cal L(X)$ over an object $X \in \Cal V$ is called an \emph {``fiber-object over $X$"}, abusing words, and denoted by $L$, $M$, etc.
\end{defn}



\begin{defn}(\cite[Definition 4.2]{Yokura-obt}) (an oriented bivariant theory) \label{orientation} Let $\bB$ be a bivariant theory on a category $\Cal V$.

\begin{enumerate}
\item For a fiber-object $L$ over $X$, the \emph{``operator" on $\bB$ associated to $L$}, denoted by $\phi(L)$, is defined to be an endomorphism
$$\phi(L): \bB(X  \xrightarrow{f}  Y) \to \bB(X  \xrightarrow{f}  Y) $$
which satisfies the following properties:

(O-1) {\bf identity}: If $L$ and $L'$ are line bundles over $X$ and isomorphic (i.e., if $f:L\to X$ and $f':L' \to X$, then there exists an isomorphism $i: L\to L'$ such that $f = f' \circ i$) , then we have
$$\phi(L) = \phi(L'): \bB(X  \xrightarrow{f}  Y) \to \bB(X  \xrightarrow{f}  Y).$$

(O-2)  {\bf commutativity}: Let $L$ and $L'$ be two fiber-objects  over $X$, then we have
$$\phi(L) \circ \phi(L') = \phi(L') \circ \phi(L) :\bB(X  \xrightarrow{f}  Y) \to \bB(X  \xrightarrow{f}  Y). $$

(O-3)   {\bf compatibility with product}: For morphisms $f:X \to Y$ and $g:Y \to Z$,  $\alp \in \bB(X  \xrightarrow{f} Y)$ and $ \be \in \bB(Y  \xrightarrow{g} Z)$,  a fiber-object  $L$ over $X$ and a fiber-object  $M$ over $Y$, we have
 $$ \phi(L) (\alp \bullet \be) = \phi(L)(\alp) \bullet \be, \quad  \phi(f^*M) (\alp \bullet \be) = \alp \bullet \phi(M)(\be).$$

(O-4)  {\bf compatibility with pushforward}: For a confined morphism $f:X \to Y$ and a fiber-object $M$ over $Y$ we have 
$$ f_*\left (\phi(f^*M)(\alp) \right ) = \phi(M)(f_*\alp).$$

(O-5)   {\bf compatibility with pullback}: For an independent square and a fiber-object  $L$ over $X$
$$
\CD
X' @> g' >> X \\
@V f' VV @VV f V\\
Y' @>> g > Y 
 \endCD
$$
we have 
$$g^*\left (\phi(L)(\alp) \right ) = \phi({g'}^*L)(g^*\alp).$$

The above operator is called an ``{\it orientation}" and a bivariant theory equipped with such an orientation is called an {\it oriented bivariant theory}, denoted by $\bOB$. 
\item An {\it oriented Grothendieck transformation} between two oriented bivariant theories is a Grothendieck transformation which preserves or is compatible with the operator, i.e., for two oriented bivariant theories $\bOB$ with an orientation $\phi$ and $\bOB'$ with an orientation $\phi'$ the following diagram commutes
$$
\CD
\bOB (X  \xrightarrow{f}  Y)  @> {\phi(L)}>> \bOB (X  \xrightarrow{f}  Y) \\
@V \ga VV @VV \ga V\\
\bOB' (X  \xrightarrow{f}  Y) @>>{\phi'(L)} > \bOB' (X  \xrightarrow{f}  Y).
 \endCD
$$
\end{enumerate}
\end{defn} 

\begin{thm} (\cite[Theorem 4.6]{Yokura-obt}) \label{obt}(A universal oriented bivariant theory)
 Let  $\Cal V$ be a category with a class $\Cal C$ of confined morphisms, a class of independent squares, a class  $\Cal S$ of specialized morphisms and $\Cal L$ a fibered category over $\Cal V$.  We define 
$$\bOM^{\Cal C} _{\Cal S}(X  \xrightarrow{f}  Y)$$
to be the free abelian group generated by the set of isomorphism classes of cobordism cycles over $X$
$$[V  \xrightarrow{h} X; L_1, L_2, \cdots, L_r]$$
such that $h \in \mathcal C$, $f \circ h: W \to Y \in \Cal S$ and $L_i$ a fiber-object over $V$.
\begin{enumerate}
\item The association $\bOM^{\Cal C} _{\Cal S}$ becomes an oriented bivariant theory if the four operations are defined as follows:
\begin{enumerate}

\item  {\bf Orientation $\Phi$}: For a morphism $f:X \to Y$ and a fiber-object $L$ over $X$, the operator
$$\phi (L):\bOM^{\Cal C} _{\Cal S} ( X  \xrightarrow{f}  Y) \to \bOM^{\Cal C} _{\Cal S} ( X  \xrightarrow{f}  Y) $$
is defined by
$$\phi(L)([V  \xrightarrow{h} X; L_1, L_2, \cdots, L_r]):=[V  \xrightarrow{h} X; L_1, L_2, \cdots, L_r, h^*L].$$
and extended linearly.

\item {\bf Product}: For morphisms $f: X \to Y$ and $g: Y
\to Z$, the product operation
$$\bullet: \bOM^{\Cal C} _{\Cal S} ( X  \xrightarrow{f}  Y) \otimes \bOM^{\Cal C} _{\Cal S} ( Y  \xrightarrow{g}  Z) \to
\bOM^{\Cal C} _{\Cal S} ( X  \xrightarrow{gf}  Z)$$
is  defined as follows: The product is defined by
\begin{align*}
& [V  \xrightarrow{h} X; L_1, \cdots, L_r]  \bullet [W  \xrightarrow{k} Y; M_1, \cdots, M_s] \\
& :=  [V'  \xrightarrow{h \circ {k}''}  X; {{k}''}^*L_1, \cdots,{{k}''}^*L_r, (f' \circ {h}')^*M_1, \cdots, (f' \circ {h}')^*M_s ]\
\end{align*}
and extended bilinearly. Here we consider the following fiber squares
$$\CD
V' @> {h'} >> X' @> {f'} >> W \\
@V {{k}''}VV @V {{k}'}VV @V {k}VV\\
V@>> {h} > X @>> {f} > Y @>> {g} > Z .\endCD
$$
\item {\bf Pushforward}: For morphisms $f: X \to Y$
and $g: Y \to Z$ with $f$ confined, the pushforward operation
$$f_*: \bOM^{\Cal C} _{\Cal S} ( X  \xrightarrow{gf} Z) \to \bOM^{\Cal C} _{\Cal S} ( Y  \xrightarrow{g}  Z) $$
is  defined by
$$f_*\left ([V  \xrightarrow{h} X; L_1, \cdots, L_r]  \right) := [V  \xrightarrow{f \circ h} Y; L_1, \cdots, L_r]$$
and extended linearly.

\item {\bf Pullback}: For an independent square
$$\CD
X' @> g' >> X \\
@V f' VV @VV f V\\
Y' @>> g > Y, \endCD
$$
the pullback operation
$$g^*: \bOM^{\Cal C} _{\Cal S} ( X  \xrightarrow{f} Y) \to \bOM^{\Cal C} _{\Cal S}( X'  \xrightarrow{f'} Y') $$
is  defined by
$$g^*\left ([V  \xrightarrow{h} X; L_1, \cdots, L_r] \right):=  [V'  \xrightarrow{{h}'}  X'; {g''}^*L_1, \cdots, {g''}^*L_r]$$
and extended linearly, where we consider the following fiber squares:
$$\CD
V' @> g'' >> V \\
@V {h'} VV @VV {h}V\\
X' @> g' >> X \\
@V f' VV @VV f V\\
Y' @>> g > Y. \endCD
$$
\end{enumerate}

\item Let $\Cal {OBT}$ be a class of oriented bivariant theories $\bOB$ on the same category $\Cal V$ with a class $\Cal C$ of confined morphisms, a class of independent squares, a class $\Cal S$ of specialized morphisms and a fibered category $\Cal L$ over $\Cal V$. Let $\Cal S$ be \underline {nice} canonical $\bOB$-orientable for any oriented bivariant theory $\bOB \in \Cal {OBT}$. Then, for each oriented bivariant theory $\bOB \in \Cal {OBT}$ with an orientation $\phi$  there exists a unique oriented Grothendieck transformation
$$\ga_{\bOB} : \bOM^{\Cal C} _{\Cal S} \to \bOB$$
such that for any $f: X \to Y \in \Cal S$ the homomorphism
$\ga_{\bOB} : \bOM^{\Cal C} _{\Cal S}(X  \xrightarrow{f}  Y) \to \bOB(X  \xrightarrow{f}  Y)$
satisfies the normalization condition that $$\ga_{\bOB}([X  \xrightarrow{\op {id}_X}  X; L_1, \cdots, L_r]) = \phi(L_1) \circ \cdots \circ \phi(L_r) (\theta_{\bOB}(f)).$$
\end{enumerate}
\end{thm}

In the following we consider the \emph{over category} or \emph{slice category} $\mathcal V/S$ of the category $\mathcal V $ over an object $S \in \mathcal V$, namely,
\begin{itemize}
\item objects of $\mathcal V/S$ are morphisms $\pi_X:X \to S$ (morphisms of $\mathcal V$),
\item morphisms of $\mathcal V/S$ are morphisms $f:X \to Y$ (morphisms of $\mathcal V$) from $\pi_X:X \to S$ to $\pi_Y:Y \to S$ such that the following diagram commutes:
$$\xymatrix{
X\ar[dr]_ {\pi_X}\ar[rr]^ {f} && Y \ \ar[dl]^{\pi_Y}\\
&S.}$$
\end{itemize}
The following proposition is a generalization of \cite[Proposition 2.3]{Yokura-obt} and its proof is done in the same way as in the proof of \cite[Proposition 2.3]{Yokura-obt}, using the axioms $A_{12}, A_{23}$ and $A_{123}$, so omitted, or left for the reader. 

\begin{pro} \label{external}Let $\bB$ a commutative bivariant theory and we restrict $\bB$ to the over category $\mathcal V/S$. 
Let us define the exterior product 
$$\times_S : \bB(X \xrightarrow {\pi_X} S ) \times \bB(Y  \xrightarrow{\pi_Y}  S) \to \bB( X \times_S Y \xrightarrow {\pi_X \times_S \pi_Y} S)$$
by 
$\alp \times_S \be := \pi_Y^*\alp \bullet \be (= \pi_X^*\be \bullet \alp)$ (note that $\bB$ is commutative), 
where we consider the following independent square
$$
\CD
X \times_S Y  @> p_2 >> Y\\
@Vp_1VV   @VV {\pi_Y} V \\
X @>> {\pi_X} > S.
\endCD
$$
The morphism  $\pi_X \times_S \pi_Y: X \times_S Y \to S$ is the composite $\pi_Y \circ p_2 = \pi_X \circ p_1$. 

\begin{enumerate}

\item For two confined morphisms $f:X_1 \to X_2$ from $\pi_{X_1}:X_1\to S$ to $\pi_{X_2}:X_2 \to S$ and $g:Y_1 \to Y_2$ from $\pi_{Y_1}:Y_1\to S$ to $\pi_{Y_2}:Y_2 \to S$, we have that for $\alp_1 \in \bB(X_1 \xrightarrow {\pi_{X_1}} S ) $ and $\be_1 \in \bB(Y_1 \xrightarrow {\pi_{Y_1}} S )$ 
$$f_*\alp_1 \times_S g_*\be_1 = (f \times _S  g)_* (\alp_1 \times_S \be_1),$$
namely,  the following diagram commutes:
$$
\CD
\bB(X_1 \xrightarrow {\pi_{X_1}} S )  \times \bB(Y_1 \xrightarrow {\pi_{Y_1}} S )   @> \times_S >> \bB( X_1 \times_S Y_1\xrightarrow {\pi_{X_1} \times_S \pi_{Y_1}} S)\\
@Vf_* \times g_*VV   @VV (f\times_S g)_* V \\
\bB(X_2 \xrightarrow {\pi_{X_2}} S )  \times \bB(Y_2 \xrightarrow {\pi_{Y_2}} S )  @>> \times _S > \bB( X_2 \times_S Y_2\xrightarrow {\pi_{X_2} \times_S \pi_{Y_2}} S).
\endCD
$$
\item If the above two morphisms $f$ and $g$ are specialized, then for
$\alp_2 \in \bB(X_2 \xrightarrow {\pi_{X_2}} S ) $ and $\be_2 \in \bB(Y_2 \xrightarrow {\pi_{Y_2}} S )$ , we have
$$f^*\alp_2 \times_S g^*\be_2 = (f \times _S  g)^* (\alp_2 \times_S \be_2),$$
namely,  the following diagram commutes:
$$
\CD
\bB(X_2 \xrightarrow {\pi_{X_2}} S )  \times \bB(Y_2 \xrightarrow {\pi_{Y_2}} S )   @> \times_S >> \bB( X_2 \times_S Y_2\xrightarrow {\pi_{X_2} \times_S \pi_{Y_2}} S)\\
@Vf^* \times g^*VV   @VV (f\times_S g)^* V \\
\bB(X_1 \xrightarrow {\pi_{X_1}} S )  \times \bB(Y_1 \xrightarrow {\pi_{Y_1}} S )  @>> \times _S > \bB( X_1 \times_S Y_1\xrightarrow {\pi_{X_1} \times_S \pi_{Y_1}} S).
\endCD
$$
\end{enumerate}

Here we use the following independent squares:
$$\CD
X_1 \times_S Y_1 @> {f''} >> X_2 \times_S Y_1 @> {p_2'} >>  Y_1\\
@V {g''}VV @VV {g'}V @VV {g}V\\
X_1 \times_S Y_2 @> {f'} >> X_2 \times_S Y_2 @> {p_2} >> Y_2 \\
@V {p_1'}VV @VV {p_1}V @VV {\pi_{Y_2}}V\\
X_1  @>> {f} > X_2  @>> {\pi_{X_2}}> S .\endCD
$$
And $\pi_{X_i} \times_S \pi_{Y_i}$ is any composite of the morphisms from $X_i \times_S Y_i$ to $S$, e.g., $\pi_{X_1} \times_S \pi_{Y_1}= \pi_{Y_2} \circ g \circ p_2' \circ f''$ and $f \times _S g :X_1 \times_S Y_1 \to X_2 \times_S Y_2$ is the composite $g' \circ f'' = r' \circ g''$.
\end{pro}

\begin{rem} If $S$ is a terminal object $pt$ (e.g., a point $pt$ in the category of topological spaces, complex algebraic varieties, schemes, etc.), then the over category  $\mathcal V/S =\mathcal V/pt$ is the same as $\mathcal V$ and we get \cite[Proposition 2.3]{Yokura-obt}.
\end{rem}

\begin{thm}\label{obt-overcategory} Let $\bB$ a commutative bivariant theory and we restrict $\bB$ to the over category $\mathcal V/S$. Then we have the following.

\begin{enumerate}

\item[(D1)'] Let $\Cal {AB}$ be the category of abelian groups. Then on the subcategory $\Cal V'/S \subset \Cal V/S$ of confined morphisms, $\bB^*(- \xrightarrow{\pi_{-}} S): \Cal V'/S \to \Cal {AB}$ is a covariant functor.
Here, for a confined morphism $f: X \to Y$ from $\pi_X:X \to S$ to $\pi_Y:Y \to S$, 
the pushforward
$$f_*: \bB^*( X  \xrightarrow{\pi_X} S) \to \bB^*( Y  \xrightarrow{\pi_Y} S) $$
defined by the bivariant pushforward.

\item[(D2)] For a specialized morphism $f: X \to Y$ from $\pi_X:X \to S$ to $\pi_Y:Y \to S$, 
the pullback
$$f^*: \bB^*( Y  \xrightarrow{\pi_Y} S) \to \bB^*( X  \xrightarrow{\pi_X} S) $$
is defined by
$f^*(\alp):= \theta(f) \bullet \alp.$

\item[(D3)] For a fiber-object $L$ over $X$, we have the operator
$$\phi (L):\bB^*( X  \xrightarrow{\pi_X}  S) \to \bB^*( X  \xrightarrow{\pi_X}  S). $$

\item[(D4)] The above external product 
$$\times_S: \bB^*( X  \xrightarrow{\pi_X} S) \times \bB^*( Y  \xrightarrow{\pi_Y} S) \to \bB^*( X \times_SY  \xrightarrow{\pi_X \times_S \pi_Y} S)$$
 is commutative, associative and admits $1 \in \bB^*(S \xrightarrow {\op{id}_S} S)$.
 
\item[(A1)] For specilaized morphisms $f: X \to Y$ from $\pi_X:X \to S$ to $\pi_Y:Y \to S$ and $g:Y \to Z$ from $\pi_Y:Y \to S$ to $\pi_Z:Z \to S$, we have
$$(g \circ f)^* = f^* \circ g^*: \bB^*(Z \xrightarrow {\pi_Z} S) \to \bB^*(X \xrightarrow {\pi_X} S) .$$

\item[(A2)] For an independent square
$$\CD
W@> {g'}>> X \\
@V {f'} VV @VV f V\\
Y@>> g > Z \endCD
$$ 
where $f \in \mathcal C$ is confined  and $g \in \mathcal S$ is specialized , we have that $g^*\circ f_* = (f')_*(g')^*$, 
i.e. the following diagram commutes:
$$\CD
\bB^*( X  \xrightarrow{\pi_X} S) @> {{g'}^*}>> \bB^*( W  \xrightarrow{\pi_W} S)  \\
@V {f_*} VV @VV {f'}_* V\\
\bB^*( Z  \xrightarrow{\pi_Z} S) @>> {g^*} > \bB^*( Y  \xrightarrow{\pi_Y} S) . \endCD
$$ 

\item[(A3)] For a confined morphism $f:X \to Y$ from $\pi_X:X \to S$ to $\pi_Y:Y \to S$ and a fiber-object  $M$ over $Y$, 
$$ f_* \circ \phi(f^*M) = \phi (M) \circ f_*: \bB^*( X  \xrightarrow{\pi_X} S) \to \bB^*( Y  \xrightarrow{\pi_Y} S) .$$

\item[(A4)] For a specialized morphism $f:X \to Y$ from $\pi_X:X \to S$ to $\pi_Y:Y \to S$ and a fiber-object $M$ over $Y$, 
$$ \phi(f^*M) \circ f^* = f^* \circ \phi (M): \bB^*( Y  \xrightarrow{\pi_Y} S) \to \bB^*( X  \xrightarrow{\pi_X} S) .$$

\item[(A5)] Let $L$ and $L'$ be two fiber-objects over $X$, then we have
$$\phi (L) \circ \phi(L') = \phi(L') \circ \phi(L) :\bB^*( X  \xrightarrow{\pi_X} S)  \to \bB^* ( X  \xrightarrow{\pi_X} S). $$
Moreover,  if $L$ and $L'$ are isomorphic, then we have that $\phi(L) = \phi(L')$.
\\
\item[(A6)] For proper morphisms $f:X_1 \to X_2$ from $\pi_{X_1}:X_1 \to S$ to $\pi_{X_2}:X_2 \to S$ and $g:Y_1 \to Y_2$ from $\pi_{Y_1}:Y_1 \to S$ to $\pi_{Y_2}:Y_2 \to S$, and $\alp \in \bB^*( X_1  \xrightarrow{\pi_{X_1}} S)$ and $\be \in \bB^*(Y_1  \xrightarrow{\pi_{Y_1}} S)$,
we have
$$f_*\alp \times_S g_*\be = (f \times_S g)_*(\alp \times_S \be).$$

\item[(A7)] For smooth morphisms $f:X_1 \to X_2$ from $\pi_{X_1}:X_1 \to S$ to $\pi_{X_2}:X_2 \to S$ and $g: Y_1 \to Y_2$ from $\pi_{Y_1}:Y_1 \to S$ to $\pi_{Y_2}:Y_2 \to S$,
$\alp \in \bB^*(X_2  \xrightarrow{\pi_{X_2}} S)$ and $\be \in \bB^*(Y_2  \xrightarrow{\pi_{Y_2}} S)$,
we have
$$f^*\alp \times_S g^*\be = (f \times_S g)^*(\alp \times_S \be).$$

\item[(A8)] For a fiber-object  $L$ over $X$ and $\alp \in \bB^*( X  \xrightarrow{\pi_X} S) $ and $\be \in \bB^*( Y  \xrightarrow{\pi_Y} S) $, 
 $$\phi(L) (\alp) \times_S \be = \phi({p_1}^*L)(\alp \times_S \be) \in \bB^*( X \times_SY  \xrightarrow{\pi_X \times_S \pi_Y} S) .$$ 
 Here we use the fiber square
 $$\CD
X \times_S Y @> {p_2}>> Y \\
@V {p_1} VV @VV {\pi_Y} V\\
X@>> {\pi_X} > S. \endCD
$$ 

\end{enumerate}
\end{thm}
We also point out that the contravaiant functor $\bB^*(X) = \bB^*(X \xrightarrow{\op{id}_X} X)$ have similar properties as above:

\begin{pro} Let $\bB$ a commutative bivariant theory on the category $\mathcal V$ with $\Cal C, \Cal S, \Cal L$ as above.

\begin{enumerate}

\item[(D1)'] For confined and specialized morphisms in $\Cal C \cap \Cal S$, i.e., for confined and nice canonical $\bB$-orientable morphisms
$f:X \to Y$, the Gysin (pushforward) homomorphisms 
$$f_*: \bB^*(X) \to \bB^*(Y)$$
are covariantly functorial.

\item[(D2)] For \underline {any} morphisms $f:X \to Y$,  we have the pullback homomorphism:
$$f^*: \bB^*(Y) \to \bB^*(X)$$

\item[(D3)] For a fiber-object $L$ over $X$ we have the operator
$$ \phi(L):\bB^*(X) \to \bB^*(X).$$

\item[(D4)] We have the exterior product
$$\times: \bB^*(X)  \times \bB^*(Y) \to \bB^*(X \times Y)$$
which is commutative, associative and admits $1 \in \bB^*(X)$.

\item[(A1)] For any morphisms $f:X \to Y$ and $g:Y \to Z$ we have
$$(g \circ f)^* = f^* \circ g^*.$$

\item[(A2)] For an independent square
$$
\CD
X' @> g' >> X \\
@V f' VV @VV f V\\
Y' @>> g > Y 
 \endCD
$$
with $g \in \Cal C \cap \Cal S$, we have $f^* \circ g_* = g'_* \circ (f')^*$, i.e., the following diagram commutes:
$$
\CD
\bB^*(Y') @> {f'}^*>> \bB^*(X') \\
@V g_* VV @VV {g'}_* V\\
\bB^*(Y) @>> f^* > \bB^*(X), 
 \endCD
$$

\item[(A3)] For a confined and specialized morphism $f: X \to Y$ and a fiber-object  $M$ over $Y$, we have
$$f_* \circ \phi (f^*M) = \phi (M) \circ f_*: \bOB^*(X) \to \bOB^*(Y).$$

\item[(A4)] For any morphism $f: X \to Y$ and fiber-object  $M$ over $Y$, we have
$$ \phi (f^*M) \circ f^* = f^*\circ \phi(M) : \bOB^*(Y) \to \bOB^*(X).$$

\item[(A5)]  Let $L$ and $L'$ be two fiber-objects  over $X$, then we have
$$\phi (L) \circ \phi(L') = \phi(L') \circ \phi(L) :\bOB^*(X) \to \bOB^*(X),$$
and if $L$ and $L'$ are isomorphic, then we have that $\phi(L) = \phi (L')$.

\item[(A6)] For both confined and specialized morphisms $f:X_1 \to X_2$ and $g:Y_1 \to Y_2$, and $\alp \in \bOB^*(X_1)$ and $\be \in \bOB^*(Y_1)$ we have
we have
$$f_*\alp \times g_*\be = (f \times g)_* (\alp \times \be).$$

\item[(A7)] For any morphisms $f:X_1 \to X_2$ and $g:Y_1\to Y_2$,and $\alp \in \bOB^*(X_2)$ and $\be \in \bOB^*(Y_2)$ 
we have
$$f^*\alp \times g^*\be = (f \times g)^* (\alp \times \be).$$

\item[(A8)] For fiber-object  $L$ over $X$ and for $\alp \in \bOB^*(X)$ and $\be \in \bOB^*(Y)$, we have
 $$\phi(L) (\alp) \times \be = \phi({p_1}^*L)(\alp \times \be) .$$
Here $p_1: X \times Y \to X$ is the projection.
\end{enumerate}

\end{pro}
\begin{rem} {} 
\begin{enumerate} 
\item Let $S=pt$. Then  in the above proposition $\bB^*( X  \xrightarrow{\pi_X} S) =\bB^*( X  \xrightarrow{p_X} pt) $ is replaced by $\bB_*(X)$ and all these properties exactly correspond to the properties $(D1), \cdots (D4)$ and $(A1), \cdots, (A8)$ of \emph{Levine-Morel's oriented Borel--Moore functor with products} \cite{LM}, except $(D1)$ which requires that the functor $\bB_*(X)$ is supposed to be additive, i.e., $\bB_*(X \sqcup Y) = \bB_*(X) \oplus \bB_*(Y).$  In order to deal with such requirements, in \cite[Rwemark 2.5]{Yokura-obt} we introduce the notion of \emph{additive bivariant theory}. 
\item Our oriented bivariant theory is a kind of \emph{bivariant-theoretic generalization} of Levine-Morel's oriented Borel--Moore functor with products. In fact, as shown in \cite[Proposition 2.4]{Yokura-obt}, having observed that given a bivariant theory $\bB$ both the covariant functor $\bB_*(X)$ and the contravariant functor $\bB^*(X)$ satisfy the properties $(D1)',(D2), (D4)$, $(A1), (A2), (A6), (A7)$ (having nothing to do with orientation $\phi(L)$) was a motivation for introducing an oriented bivariant theory so that the oriented bivariant theory satisfies the other properties $(D3), (A3), (A4), (A5), (A8)$ involving the orientation $\phi(L)$. 
\item Even in the case of an oriented bivariant theory $\mathbb {OB}$ for a general situation, the special functors $\mathbb{OB}^*(X \xrightarrow {\pi_X} S)$ (of $\mathcal V/S$), $\mathbb {OB}_*(X)$ and $\mathbb {OB}_*(X)$ shall be also called oriented Borel--Moore functors with products.
\end{enumerate}
\end{rem}

\begin{cor} The abelian group ${\bOM^{\Cal C} _{\Cal S}}_*(X):= {\bOM^{\Cal C} _{\Cal S}}(X \to pt)$
is the free abelian group generated by the set of  isomorphism classes of cobordism cycles
$$[V  \xrightarrow{h_V} X;L_1, \cdots, L_r]$$
such that $h_V:V \to X \in \Cal C$ and $V \to pt$ is a specialized map in $\Cal S$ and $L_i$ is a fiber-object over $V$. 
The abelian group ${\bOM^{\Cal C} _{\Cal S}}^*(X):= {\bOM^{\Cal C} _{\Cal S}}(X  \xrightarrow{\op {id}_X} X)$
is the free abelian group generated by the set of  isomorphism classes of cobordism cycles
$$[V  \xrightarrow{h_V} X;L_1, \cdots, L_r]$$
such that $h_V:V \to X \in \Cal C \cap \Cal S$ and $L_i$ is fiber-object over $V$. Both functor ${\bOM^{\Cal C} _{\Cal S}}_*$ and ${\bOM^{\Cal C} _{\Cal S}}^*$  are oriented Borel--Moore functors with products in the sense of Levine--Morel.
\end{cor}

\begin{cor} (A universal oriented Borel--Moore functor with products) 
Let $\Cal {BT}$ be a class of oriented additive bivariant theories $\bB$ on the same category $\Cal V$ with a class $\Cal C$ of confined morphisms, a class of independent squares, a class $\Cal S$ of specialized maps and $\Cal L$ a fibered category over $\Cal V$. Let $\Cal S$ be \underline {nice} canonical $\bOB$-orientable for any oriented bivariant theory $\bOB \in \Cal {OBT}$. Then, for each oriented bivariant theory $\bOB \in \Cal {OBT}$ with an orientation $\phi$,
\begin{enumerate}
\item there exists a unique natural transformation of oriented Borel--Moore functors with products
$${\gamma _{\bOB}}_* : {\bOM^{\Cal C} _{\Cal S}}_* \to \bOB_*$$
such that if $\pi_X:X \to pt$ is in  $\Cal S$ 
$${\gamma _{\bOB}}_* [X  \xrightarrow{\op {id}_X} X; L_1, \cdots, L_r] = \phi (L) \circ \cdots \circ \phi (L_r) ({\pi_X}^*(1_{pt})), \,\, \text{and} $$
\item there exists a unique natural transformation of oriented Borel--Moore functors with products
$${\gamma _{\bOB}}^* : {\bOM^{\Cal C} _{\Cal S}}^* \to \bOB^*$$
such that for any object $X$ 
$${\gamma _{\bOB}}^* [X  \xrightarrow{\op {id}_X} X; L_1, \cdots, L_r] = \phi (L) \circ \cdots \circ \phi (L_r) (1_X).$$
\end{enumerate}
\end{cor}

\section{A universal oriented bivariant theory on schemes}
Now, from this section on, instead of considering a general situation we consider the category $Sch$ of schemes and the category of schemes over a fixed scheme $S$ is nothing but the over category $Sch/S$. Of course, we can consider the category $\mathcal V_{\mathbb C}$ of complex algebraic varieties and the over category $\mathcal V_{\mathbb C}/S$ over a fixed variety. In this category, a confined morphism is a proper morphism, a specialized morphism is a smooth morphism, an independent square is a fiber square or fiber product, and a fiber-object $L$ over $X$ is a line bundle over $X$.

In this setup, our universal oriented bivariant theory $\mathbb {OM}^{\mathcal C}_{\mathcal S}(X \to Y)$ shall be denoted by $\mathcal Z^*(X \to Y)$
mimicking the notation used in \cite{LM}:
\begin{defn}\label{def1} We define 
$$\Cal Z^*(X  \xrightarrow{\pi_X}  S)$$
to be the free abelian group generated by the set of isomorphism classes of cobordism cycles $[V  \xrightarrow{h} X; L_1, L_2, \cdots, L_r]$ over $X$, where $L_i \, (1 \leqq i \leqq r)$ is a line bundle over $V$, 
such that 
\begin{enumerate}
\item $h:V \to X$ is proper,
\item the composite $\pi_X \circ h: V \to S$ is smooth. 
\end{enumerate}
\end{defn}

So far we never pay attention to the grading, so we define the grading as follows:

\begin{defn}\label{def2} The grading $i$ of the graded group $\Cal Z^i(X  \xrightarrow{\pi_X}  S)$ is defined by:
$$[V \xrightarrow {h}  X; L_1, \cdots, L_r] \in \Cal Z^i (X  \xrightarrow{\pi_X} S)\Longleftrightarrow -i+r= \op{dim}(\pi_X \circ h),$$
where $\op{dim}(\pi_X \circ h)$ is the relative dimension of the smooth morphism $\pi_X \circ h$, i.e. the dimension of the (smooth) fiber of 
$\pi_X \circ h$ , which is equal to $\op{dim}V - \op{dim} S$.
\end{defn}

\begin{rem} Such a grading is due to the requirement that for $S =pt$ we want to have $\Cal Z^i(X  \xrightarrow{\pi_X}  pt) = \Cal Z_{-i}(X)$ (see Definition \ref{grading}). According to the definition (\cite[Definition 2.1.6]{LM}) of grading of Levine-Morel's algebraic pre-cobordism $\Cal Z_*(X)$, the degree (or dimension) of the cobordism cycle $[V \xrightarrow {h}  X; L_1, \cdots, L_r] \in \Cal Z_*(X)$ is $\op{dim} V - r$, i.e.

$[V \xrightarrow {h}  X; L_1, \cdots, L_r] \in \Cal Z_{-i}(X) \Longleftrightarrow -i = \op{dim} V - r$, namely, $-i + r = \op{dim} V$. 
\end{rem}

Then we have the following theorem, which is just a rewritten of Theorem \ref{obt} with a bit more detailed information, in particular gradings and with a different notation for the ``orientation" $\phi(L)$, however we write it down again for the sake of reader.
\begin{thm}[\cite{Yokura-obt}] (A universal oriented bivariant theory on schemes)
\begin{enumerate}
\item The assignment $\mathcal Z^*$ becomes an oriented bivariant theory if the bivariant operations are defined as follows:

\begin{enumerate}
\item 
{\bf Orientation $\widetilde c_1$}: For a morphism $f:X \to Y$ and a line bundle $L$ over $X$, the ``Chern operator"
$$\widetilde c_1 (L):\mathcal Z^i( X  \xrightarrow{f}  Y) \to \mathcal Z^{i+1}( X  \xrightarrow{f}  Y) $$
is defined by 
$$\widetilde c_1 (L)([V  \xrightarrow{h} X; L_1, L_2, \cdots, L_r]):=[V  \xrightarrow{h_V} X; L_1, L_2, \cdots, L_r, h^*L].$$

\item {\bf Product}: For morphisms $f: X \to Y$ and $g: Y
\to Z$, the product operation
$$\bullet: \mathcal Z^i( X  \xrightarrow{f}  Y) \otimes \mathcal Z^j ( Y  \xrightarrow{g}  Z) \to
\mathcal Z^{i+j} ( X  \xrightarrow{gf}  Z)$$
is  defined as follows: The product on generators is defined by
\begin{align*}
& [V  \xrightarrow{h} X; L_1, \cdots, L_r]  \bullet [W  \xrightarrow{k} Y; M_1, \cdots, M_s] \\
& :=  [V'  \xrightarrow{h \circ k''}  X; {k''}^*L_1, \cdots,{k''}^*L_r, (f' \circ {h'})^*M_1, \cdots, (f' \circ {h'})^*M_s ],\
\end{align*}
and it extends bilinearly. Here we consider the following fiber squares
$$\CD
V' @> {h'} >> X' @> {f'} >> W \\
@V {k''}VV @V {k'}VV @V {k}VV\\
V@>> {h} > X @>> {f} > Y @>> {g} > Z .\endCD
$$

\item {\bf Pushforward}: For morphisms $f: X \to Y$
and $g: Y \to Z$ with $f$ a proper morphism, the pushforward
$$f_*: \mathcal Z^i( X  \xrightarrow{gf} Z) \to \mathcal Z^i( Y  \xrightarrow{g}  Z) $$
is  defined by
$$f_*\left ([V  \xrightarrow{h} X; L_1, \cdots, L_r]  \right) := [V  \xrightarrow{f \circ h} Y; L_1, \cdots, L_r].$$

\item {\bf Pullback}: For a fiber square \quad 
$\CD
X' @> g' >> X \\
@V f' VV @VV f V\\
Y' @>> g > Y, \endCD
$

the pullback
$$g^*:\mathcal Z^i( X  \xrightarrow{f} Y) \to \mathcal Z^i( X'  \xrightarrow{f'} Y') $$
is  defined by
$$g^*\left ([V  \xrightarrow{h} X; L_1, \cdots, L_r] \right):=  [V'  \xrightarrow{{h'}}  X'; {g''}^*L_1, \cdots, {g''}^*L_r],$$
where we consider the following fiber squares:
$$\CD
V' @> g'' >> V \\
@V {h'} VV @VV {h}V\\
X' @> g' >> X \\
@V f' VV @VV f V\\
Y' @>> g > Y. \endCD\\
$$
\end{enumerate}

\item  Let $\Cal {OBT}$ be a class of oriented bivariant theories $\bOB$ on the category $Sch$ of schemes with the same class $\Cal C$ of proper morphisms, the class of fiber squares, the class $\Cal S$ of smooth morphisms and $\mathcal L$ consisting of line bundles over $Sch$. Let $\Cal S$ be stably $\bOB$-oriented for any oriented bivariant theory $\bOB \in \Cal {OBT}$. Then, for each oriented bivariant theory $\bOB \in \Cal {OBT}$ with an orientation $\widetilde c_1$  there exists a unique oriented Grothendieck transformation
$$\ga_{\bOB} : \mathcal Z \to \bOB$$
such that for any $f: X \to Y \in \Cal S$ the homomorphism
$\ga_{\bOB} : \mathcal Z(X  \xrightarrow{f}  Y) \to \bOB(X  \xrightarrow{f}  Y)$
satisfies the normalization condition that $$\ga_{\bOB}([X  \xrightarrow{\op {id}_X}  X; L_1, \cdots, L_r]) = \widetilde c_1(L_1) \circ \cdots \circ \widetilde c_1(L_r) (\theta_{\bOB}(f)).$$
\end{enumerate}
\end{thm}

\begin{rem} Here we point out that in the product operation $\bullet$, the grading is correct. Indeed, $[V  \xrightarrow{h} X; L_1, \cdots, L_r] \in \mathcal Z^i( X  \xrightarrow{f}  Y)$ and $[W  \xrightarrow{k} Y; M_1, \cdots, M_s] \in \mathcal Z^j ( Y  \xrightarrow{g}  Z)$ imply that we have
$$ -i+r = \op{dim}(f \circ h) \quad \text{and} \quad  -j+s = \op{dim}(g \circ k).$$
Hence $-i+r +(-j+s) = \op{dim}(f \circ h) +\op{dim}(g \circ k).$ Thus we have
\begin{align*}
& -(i+j)+ (r+s)  = \op{dim}(f \circ h) +\op{dim}(g \circ k)\\
& = \op{dim}(f' \circ h') +\op{dim}(g \circ k) \quad \text{(the relative dimension is preserved by the pullback)} \\
& = \op{dim}(g \circ k \circ f' \circ h')\\
& = \op{dim}(g \circ f \circ h \circ k'') \quad \text{($g \circ k \circ f' \circ h'= g \circ f \circ h \circ k''$)}\\
& = \op{dim}((g \circ f) \circ (h \circ k'')).
\end{align*}
Thus we have
\begin{align*}
& [V  \xrightarrow{h} X; L_1, \cdots, L_r]  \bullet [W  \xrightarrow{k} Y; M_1, \cdots, M_s] \\
& =  [V'  \xrightarrow{h \circ k''}  X; {k''}^*L_1, \cdots,{k''}^*L_r, (f' \circ {h'})^*M_1, \cdots, (f' \circ {h'})^*M_s ] \in \mathcal Z^{i+j} ( X  \xrightarrow{g \circ f}  Z)
\end{align*}
\end{rem}

Then as in Proposition \ref{external} we have the following external product on the over category $Sch/S$:
$$\times_S: \Cal Z^i( X  \xrightarrow{\pi_X} S) \times \Cal Z^j( Y  \xrightarrow{\pi_Y} S) \to \Cal Z^{i+j}( X \times_SY \xrightarrow {\pi_X \times_S \pi_Y} S),$$
which we recall is defined by: for $\alp \in \Cal Z^i( X  \xrightarrow{\pi_X} S) $ and $\be \in  \Cal Z^j( Y  \xrightarrow{\pi_Y} S)$
$$\pi_Y^*(\alp) \bullet \be$$
where we consider the fiber square
$$\CD
X \times_S Y @> {p_2} >> Y \\
@V p_1 VV @VV \pi_Y V\\
X @>> \pi_X >S.\endCD
$$
More precisely, we have
\begin{align*}
& [V  \xrightarrow{h} X; L_1, \cdots, L_s]\times_S [W  \xrightarrow{k} X; M_1, \cdots, M_t]\\
& :=[V \times_S W \xrightarrow {k' \circ \widetilde h} X \times_SY; (\widetilde {p_1}\widetilde k )^*L_1, \cdots, (\widetilde {p_1}\widetilde k)^*L_s, (\widetilde {p_2}\widetilde h )^*M_1, \cdots, (\widetilde {p_2}\widetilde h)^*M_t]
\end{align*}
$$\CD
V \times _S W @> {\widetilde h}>> \widetilde W @> {\widetilde{p_2}}>> W\\
@V {\widetilde k} VV @VV {k'} V@VV k V\\
\widetilde V@>> {h'} > X \times_BY@>> {p_2}  > Y\\
@V {\widetilde{p_1}} VV @VV {p_1}  V@VV {\pi_Y} V\\
V@>> h > X@>> {\pi_X} > S.  \endCD
$$ 
If we restrict the oriented bivariant theory $\mathcal Z^*$ to the over category $Sch/S$, we have the following theorem, which is just Theorem \ref{obt-overcategory} rewritten with a bit more detailed information, in particular gradings and with a different notation for the ``orientation" $\phi(L)$, however we write it down again for the sake of reader.
\begin{thm}\label{pro1}(cf. \cite[Definition 2.1.2, Definition 2.1.10]{LM})\qquad
\begin{enumerate}
\item[(D1)] Let $\Cal {AB}$ be the category of abelian groups. Then on the subcategory $Sch'/S \subset Sch/S$ of proper morphisms, $\Cal Z^*(- \xrightarrow{\pi_{-}} S): Sch'/S \to \Cal {AB}$ is an additive functor.
Here, for a proper morphism $f: X \to Y$ from $\pi_X:X \to S$ to $\pi_Y:Y \to S$, 
the pushforward
$$f_*: \Cal Z^i( X  \xrightarrow{\pi_X} S) \to \Cal Z^i ( Y  \xrightarrow{\pi_Y} S) $$
defined by
$f_*[V  \xrightarrow{h} Y; L_1, \cdots, L_k]:=  [V  \xrightarrow{f \circ h} X; L_1, \cdots, L_k]$ is well-defined. \\
\item[(D2)] For a smooth morphism $f: X \to Y$ from $\pi_X:X \to S$ to $\pi_Y:Y \to S$, 
the pullback
$$f^*: \Cal Z^i( Y  \xrightarrow{\pi_Y} S) \to \Cal Z^{i-\op{dim}f} ( X  \xrightarrow{\pi_X} S) $$
defined by
$f^*[W  \xrightarrow{k} Y; L_1, \cdots, L_r]:=  [W'  \xrightarrow{k'} X; (f')^*L_1, \cdots, (f')^*L_r]$
is well-defined, where we use the following fiber square
$$\CD
W'@> {f'}>> W \\
@V {k'} VV @VV k V\\
X@>> f > Y, \endCD
$$ 
\item[(D3)] For a line bundle $L$ over $X$, the operator
$$\widetilde c_1 (L):\Cal Z^i( X  \xrightarrow{\pi_X}  S) \to \Cal Z^{i+1}( X  \xrightarrow{\pi_X}  S) $$
defined by 
$\widetilde c_1(L)([V  \xrightarrow{h} X; L_1, \cdots, L_k]):=[V  \xrightarrow{h} X; L_1, \cdots, L_k, h^*L]$
is well-defined.\\
\item[(D4)] The above external product 
$$\times_S: \Cal Z^i( X  \xrightarrow{\pi_X} S) \times \Cal Z^j( Y  \xrightarrow{\pi_Y} S) \to \Cal Z^{i+j}( X \times_SY  \xrightarrow{\pi_X \times_S \pi_Y} S)$$
 is commutative, associative and admits $1:=[S \xrightarrow {\op{id}_S} S] \in \Cal Z^0(S \xrightarrow {\op{id}_S} S)$.\\
\item[(A1)] For smooth morphisms $f: X \to Y$ from $\pi_X:X \to S$ to $\pi_Y:Y \to S$ and $g:Y \to Z$ from $\pi_Y:Y \to S$ to $\pi_Z:Z \to S$, we have
$$(g \circ f)^* = f^* \circ g^*: \Cal Z^i(Z \xrightarrow {\pi_Z} S) \to \Cal Z^{i -\op{dim}f -\op{dim}g}(X \xrightarrow {\pi_X} S) .$$
\item[(A2)] For a fiber square
$$\CD
W@> {g'}>> X \\
@V {f'} VV @VV f V\\
Y@>> g > Z \endCD
$$ 
where $f$ is proper and $g$ is smooth, we have that $g^*\circ f_* = (f')_*(g')^*$, 
i.e. the following diagram commutes:
$$\CD
\Cal Z^i( X  \xrightarrow{\pi_X} S) @> {{g'}^*}>> \Cal Z^{i-\op{dim}g'}( W  \xrightarrow{\pi_W} S)  \\
@V {f_*} VV @VV {f'}_* V\\
\Cal Z^i( Z  \xrightarrow{\pi_Z} S) @>> {g^*} > \Cal Z^{i-\op{dim}g}( Y  \xrightarrow{\pi_Y} S) . \endCD
$$ 
Here we note that $\op{dim}g = \op{dim}g'$.
\\
\item[(A3)] For a proper morphism $f:X \to Y$ from $\pi_X:X \to S$ to $\pi_Y:Y \to S$ and a line bundle $M$ over $Y$, 
$$ f_* \circ \widetilde c_1(f^*M) = \widetilde c_1(M) \circ f_*: \Cal Z^i( X  \xrightarrow{\pi_X} S) \to \Cal Z^{i+1}( Y  \xrightarrow{\pi_Y} S) .$$

\item[(A4)] For a smooth morphism $f:X \to Y$ from $\pi_X:X \to S$ to $\pi_Y:Y \to S$ and a line bundle $M$ over $Y$, 
$$ \widetilde c_1(f^*M) \circ f^* = f^* \circ \widetilde c_1(M): \Cal Z^i( Y  \xrightarrow{\pi_Y} S) \to \Cal Z^{i+1-\op{dim}f}( X  \xrightarrow{\pi_X} S) .$$

\item[(A5)] Let $L$ and $L'$ be two line bundles over $X$, then we have
$$\widetilde c_1(L) \circ \widetilde c_1(L') = \widetilde c_1(L') \circ \widetilde c_1(L) :\Cal Z^i( X  \xrightarrow{\pi_X} S)  \to \Cal Z^{i+2} ( X  \xrightarrow{\pi_X} S). $$
Moreover,  if $L$ and $L'$ are isomorphic, then we have that $\widetilde c_1(L) = \widetilde c_1(L')$.

\item[(A6)] For proper morphisms $f:X_1 \to X_2$ from $\pi_{X_1}:X_1 \to S$ to $\pi_{X_2}:X_2 \to S$ and $g:Y_1 \to Y_2$ from $\pi_{Y_1}:Y_1 \to S$ to $\pi_{Y_2}:Y_2 \to S$, and $\alp \in \Cal Z^*( X_1  \xrightarrow{\pi_{X_1}} S)$ and $\be \in \Cal Z^*(Y_1  \xrightarrow{\pi_{Y_1}} S)$,
we have
$$f_*\alp \times_S g_*\be = (f \times_S g)_*(\alp \times_S \be).$$

\item[(A7)] For smooth morphisms $f:X_1 \to X_2$ from $\pi_{X_1}:X_1 \to S$ to $\pi_{X_2}:X_2 \to S$ and $g: Y_1 \to Y_2$ from $\pi_{Y_1}:Y_1 \to S$ to $\pi_{Y_2}:Y_2 \to S$,
$\alp \in \Cal Z^*(X_2  \xrightarrow{\pi_{X_2}} S)$ and $\be \in \Cal Z^*(Y_2  \xrightarrow{\pi_{Y_2}} S)$,
we have
\item[(A8)] For a line bundle $L$ and $\alp \in \Cal Z^i( X  \xrightarrow{\pi_X} S) $ and $\be \in \Cal Z^j( Y  \xrightarrow{\pi_Y} S) $, 
 $$\widetilde c_1(L) (\alp) \times_S \be = \widetilde c_1({p_1}^*L)(\alp \times_S \be) \in \Cal Z^{i+j+1}( X \times_SY  \xrightarrow{\pi_X \times_S \pi_Y} S) .$$
\end{enumerate}
\end{thm}
\begin{rem} We just make a remark on (D2). Let us consider the following commutative diagram
$$\xymatrix{
W' \ar[d]_{k'} \ar[rr]^{f'} && W\ar[d]^k\\ 
X\ar[dr]_ {\pi_X}\ar[rr]^ {f} && Y \ \ar[dl]^{\pi_Y}\\
&S.}$$
\begin{align*}
\pi_X \circ k' & = (\pi_Y \circ f) \circ k' \\
& = \pi_Y \circ (f \circ k') \\
& = \pi_Y \circ (k \circ f')\\
& = (\pi_Y \circ k) \circ f'.
\end{align*}
Since $f'$ is smooth as the pullback of the smooth map $f$ (which is the given condition) and $\pi_Y \circ k$ is smooth, $\pi_X \circ k'$ is smooth. As to the grading, we can see it as follows. Suppose that $[W  \xrightarrow{k} Y; L_1, \cdots, L_r] \in \Cal Z^i( Y  \xrightarrow{\pi_Y} S)$, i.e.,  $-i+r = \op{dim} (\pi_Y \circ k)$. From which we get
\begin{align*}
-i+r + \op{dim}(f) & = \op{dim} (\pi_Y \circ k) + \op{dim}(f)\\
& = \op{dim} (\pi_Y \circ k) + \op{dim}(f') \quad \text{(since $\op{dim}(f) = \op{dim}(f')$)}\\
& = \op{dim} (\pi_Y \circ k \circ f') \\
& = \op{dim}(\pi_Y \circ f \circ k')\\
& = \op{dim}((\pi_Y \circ f) \circ k'))\\
& = \op{dim}(\pi_X \circ k').
\end{align*}
Namely we have that $-(i - \op{dim}(f)) +r = \op{dim}(\pi_X \circ k')$, which implies that
$$[W'  \xrightarrow{k'} X; (f')^*L_1, \cdots, (f')^*L_r] \in \Cal Z^{i-\op{dim}(f)}( X  \xrightarrow{\pi_X} S).$$
\end{rem}

\begin{rem}
We note that if $S$ is a point $pt =\op{Spec} k$, $\Cal Z^{-i}(X \xrightarrow {p_X} pt)$ is Levine--Morel's oriented Borel--Moore functor with products $\Cal Z_i(X)$ on $Sch_k$ \cite[Definition 2.1.6.]{LM}.
\end{rem}
\begin{defn} A functor $A^*$ assigning $A^*( X  \xrightarrow{\pi_X} S)$ to a $S$-scheme $\pi_X:X \to S$ satisfying all the properties in the above theorem is called \emph{an oriented Borel--Moore functor with products} on the category of $S$-schemes, i.e. the over category $Sch/S$.
\end{defn}

\begin{thm}\label{universality} $\Cal Z^*(- \xrightarrow{\pi_{-}} S)$ is the universal oriented Borel--Moore functor with products on $Sch_S$ in the sense that for any oriented Borel--Moore functor $A^*(- \to S)$ with products there exists a unique natural transformation $\tau_{A_*}:\Cal Z^*(- \to S) \to A^*(- \to S)$ with the requirement $\tau_{A_*}([S \xrightarrow {\op{id}_S} S]) = 1_S \in A^0(S \xrightarrow {\op{id}_S} S)$, where $1_S$ is the unit. 
\end{thm}
\begin{proof}\label{Rem} Let $[V \xrightarrow h X, L_1, \cdots L_r] \in \Cal Z^i(X \xrightarrow {\pi_X} S)$. Noticing that 
$$(\pi_X \circ h)^*([S \xrightarrow {\op{id}_S} S]) = [V \xrightarrow {\op{id}_V} V] \in \Cal Z^{- \op{dim}(\pi_X \circ h)}(V \xrightarrow {\pi_X \circ h} S),$$
 we get the following equality:
\begin{align*}
[V \xrightarrow h X, L_1, \cdots L_r] & = \widetilde {c_1}(L_1) \circ \dots \circ \widetilde {c_1}(L_r) \circ h_* ([V \xrightarrow {\op{id}_V} V]) \\
& = \widetilde {c_1}(L_1) \circ \dots \circ \widetilde {c_1}(L_r) \circ h_*\circ (\pi_X \circ h)^*([S \xrightarrow {\op{id}_S} S])
\end{align*}
Then we can define the homomorphism $\tau_{A^*}:\Cal Z^*(X\xrightarrow {\pi_X} S) \to A^*(X\xrightarrow {\pi_X} S)$ by
$$\tau_{A^*}([V \xrightarrow h X, L_1, \cdots L_r]):= \widetilde {c_1}(L_1) \circ \dots \circ \widetilde {c_1}(L_r) \circ h_*\circ (\pi_X \circ h)^* 1_S.$$
Then the uniqueness of $\tau_{A^*}$ follows from its naturality and the condition that $\tau_{A_*}([S \xrightarrow {\op{id}_S} S]) = 1_S$.
For the sake of clarity we write down the following sequence of the above homomophisms:
$$\xymatrix{
\Cal Z^0(S \xrightarrow {\op{id}_S} S) 
\ar[r]_{\tau_{A*}} 
\ar[d]_{(\pi_X \circ h)^*}  & \qquad A^0(S \xrightarrow {\op{id}_S} S) 
\ar[d]^{(\pi_X \circ h)^*} \\
\Cal Z^{- \op{dim}(\pi_X \circ h)}(V \xrightarrow {\pi_X \circ h} S) \ar[r]_{\tau_{A*}} 
\ar[d]_{h_*} & A^{- \op{dim}(\pi_X \circ h)}(V \xrightarrow {\pi_X \circ h} S) \ar[d]^{h_*}\\
\Cal Z^{- \op{dim}(\pi_X \circ h)}(X \xrightarrow {\pi_X} S) \ar[r]_{\tau_{A*}} 
\ar[d]_{\widetilde {c_1}(L_r)} & A^{- \op{dim}(\pi_X \circ h)}(X \xrightarrow {\pi_X} S) \ar[d]^{\widetilde {c_1}(L_r)} \\
\Cal Z^{- \op{dim}(\pi_X \circ h)+1}(X \xrightarrow {\pi_X} S) \ar[r]_{\tau_{A*}} 
\ar[d]_{\widetilde {c_1}(L_1) \circ \dots \circ \widetilde {c_1}(L_{r-1})} & A^{- \op{dim}(\pi_X \circ h)+1}(X \xrightarrow {\pi_X} S) \ar[d]^{\widetilde {c_1}(L_1) \circ \dots \circ \widetilde {c_1}(L_{r-1})}   \\
\Cal Z^i(X \xrightarrow {\pi_X} S) \ar[r]_{ \tau_{A^*}} & A^i(X \xrightarrow {\pi_X} S)
}
$$
Here we use the fact that $-i+r = \op{dim}(\pi_X \circ h)$.
\end{proof}

\begin{rem} It is clear that if $S$ is a point, $A_*(X) := A^{-*}(X \xrightarrow{p_X} pt)$ is Levine--Morel's oriented Borel--Moore functor with products. $\mathcal Z_*(-)$ is the universal one among all the oriented Borel--Moore functor with products in the sense that for any oriented Borel--Moore functor $A_*$ with products there exists a unique natural transformation $\tau_{A_*}:\mathcal Z_* \to A_*$ with the requirement $\tau_{A_*}([pt \xrightarrow {\op{id}_{pt}} pt]) = 1_{pt} \in A_*(pt)$, where $1_{pt}$ is the unit. The proof of this is adopted in the proof of the above theorem.
\end{rem}


\begin{rem} We note that in Levine--Morel's algebraic cobordism $\Omega_*(X)$, \emph{the particular cobordism cycle $[X \xrightarrow {\op{id}_X} X; L_1, \cdots, L_r]$ belongs to the algebraic pre-cobordism $\Cal Z_*(X)$ if and only if $X$ is smooth because of the definition of $\Cal Z_*(X)$.} 
In our case we have that $[X \xrightarrow {\op{id}_X} X; L_1, \cdots, L_r]$ belongs to $\Cal Z^*(X \xrightarrow {\pi_X} S)$ if and only if $\pi_X:X \to S$ is smooth. We also note that $[X \xrightarrow {\op{id}_X} X; L_1, \cdots, L_r]$ always belongs to $\Cal Z^*(X \xrightarrow {\op{id}_X} X)$
whether $X$ is smooth or singular.
\end{rem}
\section{Algebraic cobordism $\Omega^*(X \to S)$ of $S$-schemes}
Levine and Morel \cite{LM} construct their algebraic cobordism $\Omega_*(X)$ from the oriented Borel--More functor $\mathcal Z_*(X)$ imposing three axioms, (Dim) the dimension axiom, (Sect) the section axiom and (FGL) Formal group law axiom). In this section, from the above oriented Borel-Moore functor 
$\mathcal Z^*(X \xrightarrow {\pi_X} S)$ on the over category $Sch/S$ we construct an ``algebraic cobordism" $\Omega^*(X \to S)$ on the over category $Sch/S$ by imposing relative versions of these three axioms. In other words $\Omega^*(X \to S)$ is an algebraic cobordism of $S$-schemes.

As we will see later, the construction in this section does not give a bivariant-theoretic analogue of the algebraic cobordism $\Omega_*(X)$. 

First we  define the following relative analogues of Levine--Morel's definitions \cite[Definition 2.1.12, Definition 2.2.1]{LM}:
\begin{defn} Let $R_*$ be a commutative graded ring with unit. An oriented Borel--Moore functor with products $R_*$-theory $A(X \xrightarrow{ \pi_X} S)$ is one together with a graded ring homomorphism
$$\Phi: R_* \to A(S \xrightarrow {\op{id}_S} S).$$
\end{defn}

\begin{rem} In the above definition we should note that the external product on $A(S \xrightarrow {id_S} S)$ gives a ring structure.
\end{rem}

Let $\bL_*$ be the Lazard ring homologically graded and let $F_{\bL}(u,v) \in \bL_*[[u,v]]$ denote the universal formal group law.

\begin{defn} An oriented Borel--Moore functor with products $\bL_*$-theory $A$ is called ``of geometric type" if the following three axioms are satisfied:
\begin{enumerate}
\item(rel-Dim = Relative Dimension Axiom) For any smooth morphism $\pi_X:X \to S$ and any family $\{L_1, L_2, \cdots, L_n\}$ of line bundles on $X$ with $n > \op{dim}(\pi_X)$, one has
$$ \widetilde c_1(L_1) \circ \widetilde c_1(L_2) \circ \cdots \circ \widetilde c_1(L_n) (\pi_X^*1_S) =0 \in A^*(X \xrightarrow {\pi_X} S).$$
Here we note that $\pi_X^*1_S\in A^{-\op{dim}(\pi_X)}(X \xrightarrow {\pi_X} S)$.
\item (rel-Sect = Relative Section Axiom) For any smooth morphism $\pi_X:X \to S$, any line bundle $L$ over $X$ and any section $s$ of $L$ which is transverse fiberwisely (with respect to the smooth map $\pi_X \circ h$) to the zero section of $L$ with $Z:= s^{-1}(0)$, i.e., $\pi_Z:=\pi_X|_Z: Z \to S$ is smooth, one has
$$\widetilde c_1(L)(\pi_X^*1_S) = {i_Z}_*(\pi_Z^*1_S)$$
Here $i_Z: Z \to X$ is the closed immersion (note: $\pi_Z= \pi_X \circ i_Z$).
\item(rel-FGL = Relative Formal Group Law Axiom) Let $\widetilde c_1:\bL_* \to A(pt \to pt)$ be the ring homomorphism giving the $\bL_*$-structure and let $F_A \in A(pt \to pt)$ be the image of the universal formal group law $F_A\in \bL_*[[u,v]]$ by $\Phi$. Then for any smooth morphism $\pi_X:X \to S$ and any pair $\{L, M \}$ of line bundles over $X$, one has
$$F_A(\widetilde c_1(L), \widetilde c_1(M))(\pi_X^*1_S) = \widetilde c_1(L \otimes M)(\pi_X^*1_S).$$
\end{enumerate}
\end{defn}

\begin{rem} In the above definitions, if the target scheme $S$ is a point, we recover Levine--Morel's original definitions.
\end{rem}

First we consider imposing the (rel-Dim) on $\Cal Z^*(X \xrightarrow{\pi_X} S)$.

\begin{defn}\label{keydef} We define the following subgroup of $\Cal Z^*(X \xrightarrow {\pi_X} S)$:
$$\langle \Cal R^{Dim} \rangle (X \xrightarrow {\pi_X} S)$$
is generated by the cobordism cycles of the form
$$[V \xrightarrow h X; \pi^*L_1, \pi^*L_2, \cdots, \pi^*L_r, M_1, \cdots M_s]$$
where
\begin{enumerate}
\item the following diagram commutes
\[\xymatrix {V \ar[r]^{h} \ar[dr]_{\pi} & X \ar[r]^{\pi_X} & S &&\\
 &S' \ar[ru]_{\nu}&}
\]
\item $\pi:V \to S'$ and $\nu:S' \to S$ are both smooth. 
\item $L_1, L_2, \cdots, L_r$ are line bundles over $S'$ and $r > \op{dim}\nu = \op{dim} S' - \op{dim} S$,
\item  $M_1, \cdots, M_s$ are line bundles over $V$.
\end{enumerate}
\end{defn}

\begin{rem} Let $L_i (1 \leqq i \leqq r)$  be any line bundle over the base scheme $S$. Then any cobordism cycle
$[V \xrightarrow h X; (\pi_X \circ h)^*L_1, (\pi_X \circ h)^*L_2, \cdots, (\pi_X \circ h)^*L_r, M_1, \cdots M_s]$ always belong to $\langle \Cal R^{Dim} \rangle(X \xrightarrow {\pi_X} S).$ Because we can consider the following obvious commutative diagram
\[\xymatrix {V \ar[r]^{h} \ar[dr]_{\pi_X \circ h} & X \ar[r]^{\pi_X} & S &&\\
 &S \ar[ru]_{\nu=\op{id}_S}&}
 \]
 and the condition (3) above is satisfied: $r\geqq 1 > \op{dim}(\nu) = \op{dim}(\op{id}_S) = 0.$
 
\end{rem}
\begin{defn} We define the following quotient
$$\underline {\Cal Z}^*(X \xrightarrow {\pi_X} S):= \frac{\Cal Z^*(X \xrightarrow {\pi_X} S)}{\langle \Cal R^{Dim} \rangle(X \xrightarrow {\pi_X} S)}.$$
The equivalence class of $[V \xrightarrow h X; L_1, L_2, \cdots, L_k] \in \Cal Z^*(X \xrightarrow {\pi_X} S)$ in the quotient group $\underline {\Cal Z}^*(X \xrightarrow {\pi_X} S)$ shall be denoted by
$[[V \xrightarrow h X; L_1, L_2, \cdots, L_k]]$. 
\end{defn}

\begin{rem}
If the target scheme $S$ is a point, then the above $\langle \Cal R^{Dim} \rangle(X \xrightarrow {\pi_X} S)$ is equal to the subgroup $\langle \Cal R^{Dim} \rangle(X)$ defined in \cite[Lemma 2.4.2]{LM}. Therefore we have
$$\langle \Cal R^{Dim} \rangle(X \xrightarrow {p_X} pt) =\langle \Cal R^{Dim} \rangle(X), \quad \underline {\Cal Z}^{-i}(X \xrightarrow {p_X} pt) = \underline {\Cal Z}_i(X).$$ 
\end{rem}

\begin{thm}\label{thm1}
For the above group $\underline {\Cal Z}^*(X \xrightarrow {\pi_X} S)$, we define the following four operations:
\begin{itemize}
\item {\bf External product}: The external product 
$$\times_S: \underline {\Cal Z}^*(X \xrightarrow {\pi_X} S) \times \underline {\Cal Z}^*(Y \xrightarrow {\pi_Y} S) \to \underline {\Cal Z}^*(X \times_S Y \xrightarrow {\pi_X \times_S \pi_Y} S)$$
is defined by
\begin{align*}
 [[V  \xrightarrow{h} X; L_1, \cdots, &L_s]]\times_S [[W  \xrightarrow{k} X; M_1, \cdots, M_t]]\\
 &:= \Bigl [[V  \xrightarrow{h} X; L_1, \cdots, L_s]\times_S [W  \xrightarrow{k} X; M_1, \cdots, M_t] \Bigr ].
 \end{align*}
\item {\bf Pushforward}: For a proper morphism $f: X \to Y$ from $\pi_X:X \to S$ to $\pi_Y:Y \to S$, 
the pushforward
$$\underline{f_*}: \underline {\Cal Z}^i( X  \xrightarrow{\pi_X} S) \to \underline {\Cal Z}^i ( Y  \xrightarrow{\pi_Y} S) $$
is defined by
$\underline{f_*}([[V  \xrightarrow{h} Y; L_1, \cdots, L_k]]):=  [f_*([V  \xrightarrow{f \circ h} X; L_1, \cdots, L_k])]$.
\item {\bf Pullback}: sFor a smooth morphism $f: X \to Y$ from $\pi_X:X \to S$ to $\pi_Y:Y \to S$, 
the pullback
$$\underline{f^*}: \underline {\Cal Z} ^i( Y  \xrightarrow{\pi_Y} S) \to \underline {\Cal Z}^{i-\op{dim}f} ( X  \xrightarrow{\pi_X} S) $$
is defined by
$\underline{f^*}([[W  \xrightarrow{k} Y; L_1, \cdots, L_r]]):=  [f^*([W  \xrightarrow{k} Y; L_1, \cdots, L_r])]$.

\item {\bf Orientation = the Chern operator $\underline {\widetilde c_1 (L)}$}: For a line bundle $L$ over $X$, the operator
$$\underline {\widetilde c_1 (L)}:\underline {\Cal Z}^i( X  \xrightarrow{\pi_X}  S) \to \underline {\Cal Z}^{i+1}( X  \xrightarrow{\pi_X}  S) $$
is defined by 
$\underline {\widetilde c_1 (L)}([[V  \xrightarrow{h} X; L_1, \cdots, L_k]]):= [\widetilde c_1 (L)([V  \xrightarrow{h} X; L_1, \cdots, L_k])]$.
\end{itemize}
Then we have that 
\begin{enumerate}
\item The above operations are well-defined.
\item The above theory $\underline {\Cal Z}^*(X \xrightarrow {\pi_X} S)$ is an oriented Borel--More functor with products, i.e. it satisfies all the properties (D1), $\cdots$, (D4) and (A1), $\cdots$, (A8).
\item The above theory $\underline {\Cal Z}^*(X \xrightarrow {\pi_X} S)$ satisfies (rel-Dim). \\
\end{enumerate}
\end{thm}

\begin{proof} 
\begin{enumerate}

\item To show that the pushforward, the pullback and the orientation are well-defined, it suffices to show that each operation preserves $\langle \Cal R^{Dim} \rangle$, to be more precise, 
\begin{enumerate}
\item $f_*(\langle \Cal R^{Dim} \rangle(X \xrightarrow {\pi_X} S)) \subset \langle \Cal R^{Dim} \rangle(Y \xrightarrow {\pi_Y} S))$: 

Suppose that $[V \xrightarrow h X; \pi^*L_1, \cdots, \pi^*L_r, M_1, \cdots M_s] \in \langle \Cal R^{Dim} \rangle (X \xrightarrow {\pi_X} S)$ as in Definition \ref {keydef}. Then we have
$f_*([V \xrightarrow h X; \pi^*L_1, \pi^*L_2, \cdots, \pi^*L_r, M_1, \cdots M_s]) = [V \xrightarrow {f\circ h} Y; \pi^*L_1, \cdots, \pi^*L_r, M_1, \cdots M_s] \in \langle \Cal R^{Dim} \rangle(Y \xrightarrow {\pi_Y} S).$ Cf. the following commutative diagrama:
\[\xymatrix {V \ar[r]^{h} \ar[d]_{\pi} & X \ar[d]_{\pi_X} \ar[r]^f & Y \ar[dl]^{\pi_Y} \\
 S' \ar[r]_{\nu}& S }
\]

\item $f^*(\langle \Cal R^{Dim} \rangle(Y \xrightarrow {\pi_Y} S)) \subset \langle \Cal R^{Dim} \rangle(X \xrightarrow {\pi_X} S)$: Here we should note that $f:X \to Y$ is smooth, which is important. Suppose that $[W \xrightarrow k Y; \pi^*L_1, \cdots, \pi^*L_r, M_1, \cdots M_s] \in \langle \Cal R^{Dim} \rangle (Y \xrightarrow {\pi_Y} S)$ as in Definition \ref {keydef}. Consider the following commutative diagram
\[\xymatrix{
W' \ar[d]_{k'} \ar[r]^{f'} & W\ar[d]^k \ar[dr]^{\pi}&\\ 
X\ar[dr]_ {\pi_X}\ar[r]^ {f} & Y \ \ar[d]^{\pi_Y} & S' \ar[dl]^{\nu}\\
&S.}
\]
Then we have
\begin{align*}
 f^*([W \xrightarrow k Y; & \pi^*L_1,\cdots, \pi^*L_r, M_1, \cdots M_s]) \\
& = [W' \xrightarrow k' X; (f')^*\pi^*L_1, \cdots, (f')^*\pi^*L_r, (f')^*M_1, \cdots (f')^*M_s]\\
& =  [W' \xrightarrow k' X; (\pi \circ f')^*L_1, \cdots, (\pi \circ f')L_r, (f')^*M_1, \cdots (f')^*M_s]
\end{align*}
which belongs to $\langle \Cal R^{Dim} \rangle(X \xrightarrow {\pi_X} S)$.

\item $\widetilde c_1 (L)(\langle \Cal R^{Dim} \rangle(X \xrightarrow {\pi_X} S)) \subset \langle \Cal R^{Dim} \rangle(X \xrightarrow {\pi_X} S)$: It is clear.

\item As to the external product, we need to show that
\begin{enumerate}
\item $\Cal Z^*(X \xrightarrow {\pi_X} S) \times_S \langle \Cal R^{Dim} \rangle(Y \xrightarrow {\pi_Y} S) \subset \langle \Cal R^{Dim} \rangle(X \times_S Y \xrightarrow {\pi_X \times_S \pi_Y} S)$,
\item $\langle \Cal R^{Dim} \rangle(X \xrightarrow {\pi_X} S) \times_S \Cal Z^*(Y \xrightarrow {\pi_Y} S) \subset \langle \Cal R^{Dim} \rangle(X \times_S Y \xrightarrow {\pi_X \times_S \pi_Y} S)$,
\item  $\langle \Cal R^{Dim} \rangle(X \xrightarrow {\pi_X} S) \times_S \langle \Cal R^{Dim} \rangle(Y \xrightarrow {\pi_Y} S) \subset \langle \Cal R^{Dim} \rangle(X \times_S Y \xrightarrow {\pi_X \times_S \pi_Y} S).$
\end{enumerate}
Since (iii) is a special case, it suffices to show (i) and (ii). For (ii), suppose that 

$[V \xrightarrow h X; \pi^*L_1, \cdots, \pi^*L_r, M_1, \cdots M_s] \in \langle \Cal R^{Dim} \rangle (X \xrightarrow {\pi_X} S)$ as in Definition \ref {keydef} and 
$[W \xrightarrow k ; N_1, \cdots, N_t] \in \Cal Z^*(Y \xrightarrow {\pi_Y} S)$. Then we consider the following commutative diagrams:
\[\xymatrix{
V \times _S W \ar[d]_{\widetilde k} \ar[rr]^{\widetilde h} && \widetilde W \ar[d]^{k'} \ar[rr]^{\widetilde{p_2}} && W \ar[d]^k\\
\widetilde V \ar[rr]^{h'} \ar[d]_{\widetilde{p_1}} && X \times_SY \ar[d]^{p_1} \ar[rr]^{p_2} && Y \ar[d]^{\pi_Y}\\
V \ar[rr]^h \ar[drr]_{\pi} && X \ar[rr]^{\pi_X} && S\\
&& S' \ar[urr]_{\nu} &\\
}
\]
\begin{align*}
& [V  \xrightarrow{h} X; \pi^*L_1, \cdots, \pi^*L_r, M_1, \cdots M_s]\times_S [W  \xrightarrow{k} X; N_1, \cdots, N_t]\\
& =[V \times_S W \xrightarrow {k' \circ \widetilde h} X \times_SY; (\widetilde {p_1}\widetilde k )^*L_1, \cdots, (\widetilde {p_1}\widetilde k )^*\pi^*L_r, \\
& \hspace{4cm} (\widetilde {p_1}\widetilde k )^*M_1, \cdots, (\widetilde {p_1}\widetilde k)^*M_s, (\widetilde {p_2}\widetilde h )^*N_1, \cdots, (\widetilde {p_2}\widetilde h)^*N_t]\\
& =[V \times_S W \xrightarrow {k' \circ \widetilde h} X \times_SY; (\pi \widetilde {p_1}\widetilde k )^*L_1, \cdots, (\pi \widetilde {p_1}\widetilde k )^*L_r, \\
& \hspace{4cm} (\widetilde {p_1}\widetilde k )^*M_1, \cdots, (\widetilde {p_1}\widetilde k)^*M_s, (\widetilde {p_2}\widetilde h )^*N_1, \cdots, (\widetilde {p_2}\widetilde h)^*N_t].
\end{align*}
Here we note that $\pi \widetilde {p_1}\widetilde k: V \times_S W \to S'$ is smooth because $\pi$ is smooth by hypothesis and $\widetilde{p_1}\circ \widetilde k:V \times _S W \to V$ is smooth since it is the pullback of the smooth morphism $\pi_Y \circ k:W \to S$. The proof of (i) is the same as this, so omitted.

\end{enumerate}

\item (D1), $\cdots$, (D4)  are already checked above, thus it suffices to see (A1), $\cdots$, (A8). But they follow from the definitions of these four operations. E.g., as to (A1), we can see it as follows: for $[[x]] \in \underline {\Cal Z}^i( X  \xrightarrow{\pi_X} S)$, where $[x] = [V  \xrightarrow{h} X; L_1, \cdots, L_s]$,

we have
\begin{align*}
\underline{(g \circ f)^*}([[x]]) & = [(g \circ f)^*([x])] \quad \text{(by the definition)} \\
& = [(f^* \circ g^*)([x])] \\
& = \underline{f^*}([g^*([x])]) \\
& =  \underline{f^*} \circ \underline {g^*} ([[x]])
\end{align*}
Thus we have $\underline{(g \circ f)^*} = \underline{f^*} \circ \underline {g^*}.$

\item In our case, since $1_S = [S \xrightarrow {\op{id}_S} S]$ and $\pi_X^*1_S = [X \xrightarrow {\op{id}_X} X]$, we have that
$$\widetilde c_1(L_1) \circ \widetilde c_1(L_2) \circ \cdots \circ \widetilde c_1(L_n) (\pi_X^*1_S)
= [X \xrightarrow {\op{id}_X} X; L_1, L_2, \cdots L_n]$$
with $\pi_X:X \to S$ is smooth.
Then we have the following obvious commutative diagram:
\[\xymatrix {X \ar[r]^{\op{id}_X} \ar[dr]_{\op{id}_X} & X \ar[r]^{\pi_X} & S &&\\
 &X \ar[ru]_{\nu=\pi_X}&}
 \]
Since $r>\op{dim}(\pi_X) = \op{dim}(\nu)$, we have that 
$$[X \xrightarrow {\op{id}_X} X; L_1, L_2, \cdots L_n] \in \langle \Cal R^{Dim} \rangle (X \xrightarrow {\pi_X} S).$$
Therefore we have that
$$[[X \xrightarrow {\op{id}_X} X; L_1, L_2, \cdots L_n]]=0 \in \underline {\Cal Z}^*(X \xrightarrow {\pi_X} S).$$
\end{enumerate}

\end{proof}

Next we impose the axiom (rel-Sect) on the above quotient group $\underline {\Cal Z^*}  ( X \xrightarrow {\pi_X} S )$.

\begin{defn}\label{rel-Sect} We define the following subgroup of $\underline {\Cal Z^*} (X \xrightarrow {\pi_X} S )$:
$$\langle \Cal R^{Sect} \rangle(X \xrightarrow {\pi_X} S )$$
is generated by elements of the form
$$[[V \xrightarrow h X; L_1, \cdots, L_r]] - [[Z \xrightarrow {h|_Z} X; i^*L_1, \cdots, i^*L_{r-1}]],$$
where
\begin{enumerate}
\item $r>0$
\item $Z=s^{-1}(0)$, where $s$ is a section of the line bundle $L_r$ which is transverse fiberwisely (with respect to the smooth map $\pi_X \circ h$) to the zero section of $L_r$ (hence $\pi_X \circ h|_Z$ is smooth) and $i: Z \hookrightarrow X$ is the inclusion and $h|_Z = h \circ i.$
\end{enumerate}
\end{defn}

\begin{defn} We define the following quotient group 
$$\underline {\Omega}^*(X \xrightarrow {\pi_X} S ):= \frac{\underline {\Cal Z}^*(X \xrightarrow {\pi_X} S )}{\langle \Cal R^{Sect} \rangle(X \xrightarrow {\pi_X} S )}.$$
The equivalence class of $[[V \xrightarrow h X; L_1, L_2, \cdots, L_k] ]\in \underline {\Cal Z_*}(X \xrightarrow {\pi_X} S )$ in the quotient group $\underline {\Omega}^*(X \xrightarrow {\pi_X} S )$ shall be denoted by
$\left [[[V \xrightarrow h X; L_1, L_2, \cdots, L_k]] \right]$.
\end{defn}

\begin{rem}
If the target scheme $Y$ is a point, then the above $\langle \Cal R^{Sect} \rangle(X \xrightarrow f Y)$ is equal to the subgroup $\langle \Cal R^{Sect} \rangle (X)$ defined in \cite[Lemma 2.4.7]{LM}. Therefore we have
$$\langle \Cal R^{Sect} \rangle(X \xrightarrow {p_X} pt) =\langle \Cal R^{Sect} \rangle(X), \quad \underline {\Omega}^{-i}(X \xrightarrow {p_X} pt) = \underline {\Omega}_i(X).$$ 
\end{rem}

\begin{thm}\label{thm2}
For the above group $\underline {\Omega}^*(X \xrightarrow {\pi_X} S )$, we define the following four operations as follows:
\begin{itemize}
\item (external product) The external product 
$$\times_S: \underline {\Omega}^*(X \xrightarrow {\pi_X} S ) \times \underline {\Omega}^*(Y \xrightarrow {\pi_Y} S ) \to \underline {\Omega}^*(X \times_S Y \xrightarrow {\pi_X \times_S \pi_Y} S)$$
is defined by

\begin{align*}
 \left [ [[V  \xrightarrow{h} X; L_1, \cdots, L_s ]] \right ]  & \times_S \left [ [[W  \xrightarrow{k} X; M_1, \cdots, M_t]] \right ] \\
 &:= \left [ [[V  \xrightarrow{h} X; L_1, \cdots, L_s]] \times_S [[W  \xrightarrow{k} X; M_1, \cdots, M_t]] \right].
 \end{align*}
\item (pushforward) For a proper morphism $f: X \to Y$ from $\pi_X:X \to S$ to $\pi_Y:Y \to S$, 
the pushforward
$$\underline{\underline{f_*}}: \underline {\Omega}^*( X  \xrightarrow{\pi_X} S) \to \underline {\Omega}^* ( Y  \xrightarrow{\pi_Y} S) $$
is defined by
$\underline{\underline{f_*}}(\left [[[V  \xrightarrow{h} Y; L_1, \cdots, L_k]] \right ]):=  \left [\underline{f_*}([[V  \xrightarrow{f \circ h} X; L_1, \cdots, L_k]]) \right]$.

\item (pullback) For a smooth morphism $f: X \to Y$ from $\pi_X:X \to S$ to $\pi_Y:Y \to S$, 
the pullback
$$\underline{\underline{f^*}}: \underline {\Omega}^i( Y  \xrightarrow{\pi_Y} S) \to \underline {\Omega}^{i-\op{dim}f} ( X  \xrightarrow{\pi_X} S) $$
is defined by
$\underline{\underline{f^*}}(\left [[[W  \xrightarrow{k} Y; L_1, \cdots, L_r]] \right ]):=  \left [\underline {f^*}([[W  \xrightarrow{k} Y; L_1, \cdots, L_r]])\right]$.

\item (orientation = the Chern operator $\underline{\underline {\widetilde c_1 (L)}}$) For a line bundle $L$ over $X$, the operator
$$\underline{\underline {\widetilde c_1 (L)}}:\underline {\Omega}^i( X  \xrightarrow{\pi_X}  S) \to \underline {\Omega}^{i+1}( X  \xrightarrow{\pi_X}  S) $$
is defined by 
$\underline {\underline {\widetilde c_1 (L)}} (\left [[[V  \xrightarrow{h} X; L_1, \cdots, L_k]] \right ):= \left [\underline{\widetilde c_1 (L)}([[V  \xrightarrow{h} X; L_1, \cdots, L_k]]) \right]$.
\end{itemize}
\begin{enumerate}
\item The above operations are well-defined.
\item The above theory $\underline {\Omega}^*(X \xrightarrow {\pi_X} S)$ is an oriented Borel--More functor with products, i.e. it satisfies all the properties (D1), $\cdots$, (D4) and (A1), $\cdots$, (A8).
\item The above theory $\underline {\Omega}^*(X \xrightarrow {\pi_X} S)$ satisfies (rel-Dim) and (rel-Sec).
. \\
\end{enumerate}
\end{thm}

\begin{proof} 
The proof is similar to that of the above Theorem \ref{thm1}. But we will write them down for the sake of completeness.

\noindent
(1) To show that the pushforward, the pullback and the orientation are well-defined, it suffices to show that each operation preserves $\langle \Cal R^{Sect} \rangle$, to be more precise, 

(i)  $\underline{f_*}(\langle \Cal R^{Sect} \rangle(X \xrightarrow {\pi_X} S)) \subset \langle \Cal R^{Sect} \rangle(Y \xrightarrow {\pi_Y} S))$: Indeed, let

$[[V \xrightarrow h X; L_1, \cdots, L_r]] - [[Z \xrightarrow {h|_Z} X; i^*L_1, \cdots, i^*L_{r-1}]] \in \langle \Cal R^{Sect} \rangle(X \xrightarrow {\pi_X} S)$,
where $Z$, the line bundles $L_i$'s and $i:Z \to V$ are as in Definition \ref{rel-Sect} above.

Then
\begin{align*}
\underline{f_*}([[V \xrightarrow h X; & L_1, \cdots, L_r]]  - [[Z \xrightarrow {h|_Z} X; i^*L_1, \cdots, i^*L_{r-1}]]) \\
& = [ f_*([V \xrightarrow h X; L_1, \cdots, L_r])] - [f_*([Z \xrightarrow {h|_Z} X; i^*L_1, \cdots, i^*L_{r-1}])] \\
& = [[V \xrightarrow {f \circ h_X} Y; L_1, \cdots, L_r]] - [[Z \xrightarrow {f \circ h|_Z} Y; i^*L_1, \cdots, i^*L_{r-1}]],
\end{align*}
which belongs to $\langle \Cal R^{Sect} \rangle(Y \xrightarrow {\pi_Y} S)$.

(ii) $\underline{f^*}(\langle \Cal R^{Sect} \rangle(Y \xrightarrow {\pi_Y} S)) \subset \langle \Cal R^{Sect} \rangle(X \xrightarrow {\pi_X} S)$: Let 

$[[W \xrightarrow k Y; M_1, \cdots, M_r]] - [[Z \xrightarrow {k|_Z} Y; i^*M_1, \cdots, i^*M_{r-1}]] \in \langle \Cal R^{Sect} \rangle(Y \xrightarrow {\pi_Y} S)$.
Then we consider the following fiber squares:
$$\CD
Z'@> {f''}>> Z \\
@V {i'} VV @VV i V\\
W'@> {f'}>> W \\
@V {k'} VV @VV k V\\
X@>> f > Y, \endCD
$$ 
Here $i:Z \to W$ is the inclusion and thus the pullback $i':Z' \to W$ is also an inclusion and $k|_Z:Z \to Y$ is the composite $k\circ i$ and $k|_{Z'}:Z' \to Y$ is the composite $k' \circ i'$. Then we have
$$(f'')^*(i^*M_j) = (i')^*((f')^*M_j).$$
Hence we have 
\begin{align*}
& \underline{f^*}([[W \xrightarrow k Y; M_1, \cdots, M_r]] - [[Z \xrightarrow {k|_Z} Y; i^*M_1, \cdots, i^*M_{r-1}]]) \\
& = [ f^*([W \xrightarrow k Y; M_1, \cdots, M_r])] - [f^*([Z \xrightarrow {k \circ i} Y; i^*M_1, \cdots, i^*M_{r-1}])] \\
& = [[W' \xrightarrow {k'} X; (f')^*M_1, \cdots, (f')^*M_r]] - [[Z' \xrightarrow {k' \circ i'} X; (f'')^*(i^*M_1), \cdots, (f'')^*(i^*M_{r-1})]]\\
& = [[W' \xrightarrow {k'} X; (f')^*M_1, \cdots, (f')^*M_r]] - [[Z' \xrightarrow {k' \circ i'} X; (i')^*((f')^*M_1), \cdots, (i')^*((f')^*M_{r-1})]]\\
& = [[W' \xrightarrow {k'} X; (f')^*M_1, \cdots, (f')^*M_r]] - [[Z' \xrightarrow {k'|_{Z'}} X; (i')^*((f')^*M_1), \cdots, (i')^*((f')^*M_{r-1})]].
\end{align*}
Note that $Z'$ is the zero locus of the section $s': W' \to (f')^*M_r$ which is the pullback of the section $s:W \to M_r$. Hence 

$[[W' \xrightarrow {k'} X; (f')^*M_1, \cdots, (f')^*M_r]] - [[Z' \xrightarrow {k'|_{Z'}} X; (i')^*((f')^*M_1), \cdots, (i')^*((f')^*M_{r-1})]]$ 

belongs to
$\langle \Cal R^{Sect} \rangle(X \xrightarrow {\pi_X} S)$.

\noindent
(iii) $\underline{\widetilde c_1 (L)}(\langle \Cal R^{Sect} \rangle(X \xrightarrow {\pi_X} S)) \subset \langle \Cal R^{Sect} \rangle(X \xrightarrow {\pi_X} S)$: Let $L$ be a line bundle over $X$ and let 

$[[V \xrightarrow h X; L_1, \cdots, L_r]] - [[Z \xrightarrow {h|_Z} X; i^*L_1, \cdots, i^*L_{r-1}]] \in \langle \Cal R^{Sect} \rangle(X \xrightarrow {\pi_X} S)$. Then 
\begin{align*}
& \underline{\widetilde c_1 (L)}([[V \xrightarrow h X; L_1, \cdots, L_r]] - [[Z \xrightarrow {h|_Z} X; i^*L_1, \cdots, i^*L_{r-1}]]) \\
& =[\widetilde c_1 (L)([V \xrightarrow h X; L_1, \cdots, L_r])] - [\widetilde c_1 (L)([Z \xrightarrow {h|_Z} X; i^*L_1, \cdots, i^*L_{r-1}])]) \\
& = [[V \xrightarrow h X; L_1, \cdots, L_r, h^*L]] - [[Z \xrightarrow {h|_Z} X; i^*L_1, \cdots, i^*L_{r-1}, (h|_Z)^*L]]\\
& =  [[V \xrightarrow h X; L_1, \cdots, L_r, h^*L]] - [[Z \xrightarrow {h|_Z} X; i^*L_1, \cdots, i^*L_{r-1}, i^*(h^*L)]], \quad \text{(since $h|_Z = h \circ i$)}
\end{align*}
which belongs to $\langle \Cal R^{Sect} \rangle(X \xrightarrow {\pi_X} S)$.

\noindent
(iv) As to the external product, we need to see that not only 
$$\langle \Cal R^{Sect} \rangle(X \xrightarrow {\pi_X} S) \times_S \langle \Cal R^{Sect} \rangle(Y \xrightarrow {\pi_Y} S) \subset \langle \Cal R^{Sect} \rangle(X \times_S Y \xrightarrow {\pi_X \times_S \pi_Y} S)$$
but also 
$$\underline{\Cal Z^*}(X \xrightarrow {\pi_X} S) \times_S \langle \Cal R^{Sect} \rangle(Y \xrightarrow {\pi_Y} S) \subset \langle \Cal R^{Sect} \rangle(X \times_S Y \xrightarrow {\pi_X \times_S \pi_Y} S), $$
$$\langle \Cal R^{Sect} \rangle(X \xrightarrow {\pi_X} S) \times_S \underline{\Cal Z^*}(Y \xrightarrow {\pi_Y} S) \subset \langle \Cal R^{Sect} \rangle(X \times_S Y \xrightarrow {\pi_X \times_S \pi_Y} S).$$
For this, it suffices to show the second one, because the other two are similar. So, we consider 
\begin{align*}
& [[V \xrightarrow h X; L_1, \cdots, L_r]] \times_S 
\left ( [[W \xrightarrow k Y; M_1, \cdots, M_q]] - [[Z \xrightarrow {k|_Z} Y; i^*M_1, \cdots, i^*M_{q-1}]] \right ) \\
& = [[V \xrightarrow h X;  L_1, \cdots, L_r]] \times_S [[W \xrightarrow k Y; M_1, \cdots, M_q]] \\
& \hspace {3cm} - [[V \xrightarrow h X;  L_1, \cdots, L_r]] \times_S [[Z \xrightarrow {k|_Z} Y; i^*M_1, \cdots, i^*M_{q-1}]]
\end{align*}
We note that $i:Z \to W$ is the inclusion and $k|_Z$ is the composite $k \circ i$, and we recall that 
\begin{align*}
& [[V  \xrightarrow{h} X; L_1, \cdots, L_r]]\times_S [[W  \xrightarrow{k} X; M_1, \cdots, M_q]]\\
& = [ [V  \xrightarrow{h} X; L_1, \cdots, L_r] \times_S [W  \xrightarrow{k} X; M_1, \cdots, M_q]] \\
& =[[ V \times_S W \xrightarrow {k' \circ \widetilde h} X \times_SY; (\widetilde {p_1}\widetilde k )^*L_1, \cdots, (\widetilde {p_1}\widetilde k)^*L_r, (\widetilde {p_2}\widetilde h )^*M_1, \cdots, (\widetilde {p_2}\widetilde h)^*M_q]].\\
& [[V \xrightarrow h X;  L_1, \cdots, L_r]] \times_S [[Z \xrightarrow {k|_Z} Y; i^*M_1, \cdots, i^*M_{q-1}]]\\
& = [[V \xrightarrow h X;  L_1, \cdots, L_r] \times_S [Z \xrightarrow {k|_Z} Y; i^*M_1, \cdots, i^*M_{q-1}]]\\
& =[[V \times_S Z \xrightarrow {\widetilde i \circ (k' \circ \widetilde h) } X \times_SY; \widetilde i^*(\widetilde {p_1}\widetilde k )^*L_1, \cdots, \widetilde i^*(\widetilde {p_1}\widetilde k)^*L_r, (\widetilde {\widetilde{p_2}} \widetilde {\widetilde h})^*i^*M_1, \cdots, (\widetilde {\widetilde{p_2}} \widetilde {\widetilde h})^*i^*M_{q-1}]]\\
& =[[V \times_S Z \xrightarrow {\widetilde i \circ (k' \circ \widetilde h) } X \times_SY; \widetilde i^*(\widetilde {p_1}\widetilde k )^*L_1, \cdots, \widetilde i^*(\widetilde {p_1}\widetilde k)^*L_r, \widetilde i^*(\widetilde {p_2}\widetilde h )^*M_1, \cdots, \widetilde i^*(\widetilde {p_2}\widetilde h)^*M_{q-1}]],
\end{align*}
where $V \times_S Z $ is the zero locus of the section from $V \times _S W$ to the pullbacked line bundle $(\widetilde {p_2}\widetilde h)^*M_q$. So the last one belongs to $\langle \Cal R^{Sect} \rangle(X \times_S Y \xrightarrow {\pi_X \times_S \pi_Y} S)$. Here we use the diagram:
$$\CD
V \times _S Z @> {\widetilde {\widetilde h}}>> \widetilde Z @> {\widetilde {\widetilde{p_2}}}>> Z\\
@V {\widetilde i} VV @VV {i'} V@VV i V\\
V \times _S W @> {\widetilde h}>> \widetilde W @> {\widetilde{p_2}}>> W\\
@V {\widetilde k} VV @VV {k'} V@VV k V\\
\widetilde V@>> {h'} > X \times_BY@>> {p_2}  > Y\\
@V {\widetilde{p_1}} VV @VV {p_1}  V@VV {\pi_Y} V\\
V@>> h > X@>> {\pi_X} > S.  \endCD
$$

\noindent
(2) (D1), $\cdots$, (D4)  are already checked above, thus it suffices to see (A1), $\cdots$, (A8). But they follow from the definitions of these four operations.

\noindent
(3)  $\widetilde c_1(L)(\pi_X^*1_S) = (i_Z)_*(\pi_Z^*1_S)$ is nothing but $[X \xrightarrow {\op{id}_X} X; L] = [Z \xrightarrow {i} X]$. Since we have
$[[X \xrightarrow {\op{id}_X} X; L]] - [[Z \xrightarrow {i} X]] \in \langle \Cal R^{Sect} \rangle(X \xrightarrow {\pi_X} S)$, we have $\underline{\underline {\widetilde c_1(L)}}(\underline{\underline {\pi_X^*}}1_S) = \underline{\underline{{i_Z}_*}}(\underline{\underline {\pi_X^*}}1_S).$

\end{proof}

Finally we define the following
\begin{defn} We define the following subgroup of $\bL_* \otimes \underline {\Omega}^*(X \xrightarrow {\pi_X} S )$:
$$\langle \bL_*\Cal R^{FGL} \rangle(X \xrightarrow {\pi_X} S )$$
is generated by the elements of the form
$$\left [[[V \xrightarrow h X; L_1, \cdots L_r, F_{\bL}(L, M)]] \right] - \left [[[V \xrightarrow h X; L_1, \cdots L_r, L \otimes M]] \right] $$
where $L_i$, $L$ and $M$ are all line bundles over $V$.
\end{defn} 

\begin{defn} We define the following quotient group 
$$\Omega^*(X \xrightarrow {\pi_X} S ):= \frac{\bL_* \otimes \underline {\Omega}^*(X \xrightarrow {\pi_X} S )}{\langle \bL_*\Cal R^{FGL} \rangle(X \xrightarrow {\pi_X} S )}.$$
The equivalence class of $\left [[[V \xrightarrow h X; L_1, L_2, \cdots, L_k]] \right] \in \bL_* \otimes \underline {\Omega}^*(X \xrightarrow {\pi_X} S )$ in the quotient group $\Omega^*(X \xrightarrow {\pi_X} S )$ shall be denoted by
$$\left [\left [[[V \xrightarrow h X; L_1, L_2, \cdots, L_k]] \right] \right].$$
\end{defn}

\begin{thm}
For the above group $\Omega^*(X \xrightarrow {\pi_X} S )$, we define the following four operations as follows:
\begin{itemize}
\item (external product) The external product 
$$\times_S: \Omega^*(X \xrightarrow {\pi_X} S ) \times \Omega^*(Y \xrightarrow {\pi_Y} S ) \to \Omega ^*(X \times_S Y \xrightarrow {\pi_X \times_S \pi_Y} S)$$
is defined by

\begin{align*}
 \left [\left [ [[V  \xrightarrow{h} X; L_1, \cdots, L_s ]] \right ] \right] & \times_S \left [\left [ [[W  \xrightarrow{k} X; M_1, \cdots, M_t]] \right ] \right] \\
 &:= \Biggl [ \left [ [[V  \xrightarrow{h} X; L_1, \cdots, L_s]] \right ] \times_S \left [ [[W  \xrightarrow{k} X; M_1, \cdots, M_t]] \right] \Biggr ].
 \end{align*}
\item (pushforward) For a proper morphism $f: X \to Y$ from $\pi_X:X \to S$ to $\pi_Y:Y \to S$, 
the pushforward
$$f_*: \Omega^*( X  \xrightarrow{\pi_X} S) \to \Omega^* ( Y  \xrightarrow{\pi_Y} S) $$
is defined by
$f_* (\left [\left [[[V  \xrightarrow{h} Y; L_1, \cdots, L_k]] \right ] \right] ):=  \left [\underline {\underline{f_*}}(\left [ [[V  \xrightarrow{f \circ h} X; L_1, \cdots, L_k]] \right] ) \right]$.

\item (pullback) For a smooth morphism $f: X \to Y$ from $\pi_X:X \to S$ to $\pi_Y:Y \to S$, 
the pullback
$$f^*: \Omega^i( Y  \xrightarrow{\pi_Y} S) \to \Omega^{i-\op{dim}f} ( X  \xrightarrow{\pi_X} S) $$
is defined by

$f^*(\left [\left [[[W  \xrightarrow{k} Y; L_1, \cdots, L_r]] \right ] \right]):=  \left [\underline{\underline {f^*}}(\left [ [[W'  \xrightarrow{k'} X; (f')^*L_1, \cdots, (f')^*L_r]] \right ]) \right]$,
where we use the following fiber square
$$\CD
W'@> {f'}>> W \\
@V {k'} VV @VV k V\\
X@>> f > Y, \endCD
$$ 
\item (orientation = the Chern operator $\widetilde c_1 (L)$) For a line bundle $L$ over $X$, the operator
$$\widetilde c_1 (L):\Omega^i( X  \xrightarrow{\pi_X}  S) \to \Omega^{i+1}( X  \xrightarrow{\pi_X}  S) $$
is defined by 

$\widetilde c_1 (L) (\left [\left [[[V  \xrightarrow{h} X; L_1, \cdots, L_k]] \right] \right]):= \left [\underline{\underline{\widetilde c_1 (L)}}(\left [ [[V  \xrightarrow{h} X; L_1, \cdots, L_k, h^*L]] \right]) \right]$.
\end{itemize}
\begin{enumerate}
\item The above operations are well-defined.
\item The above theory $\Omega^*(X \xrightarrow {\pi_X} S)$ is an oriented Borel--More functor with products, i.e. it satisfies all the properties (D1), $\cdots$, (D4) and (A1), $\cdots$, (A8).
\item The above theory $\Omega^*(X \xrightarrow {\pi_X} S)$ satisfies (rel-Dim), (rel-Sec) and (rel-FGL).
. \\
\end{enumerate}
\end{thm}

\begin{proof} It is easy to see that as above the pushforward, the pullback and the orientation are all well-defined. As to the external product, we basically deal with pulling back line bundles, and the tensor $\otimes$ and the formal group law commute with the pullback operation, therefore we can see that the external product is also well-defined. It is also clear that it satisfies (A1), $\cdots$, (A8) and (rel-Dim), (rel-Sec) and (rel-FGL).
\end{proof}
\begin{rem}
If the target scheme $Y$ is a point, then the above $\langle \bL_*  \Cal R^{FGL} \rangle(X \xrightarrow f Y)$ is equal to the subgroup $\langle \bL_*  \Cal R^{FGL} \rangle(X)$ defined in \cite[Remark 2.4.11]{LM}. Hence we have
$$\langle \bL_*  \Cal R^{FGL} \rangle(X \xrightarrow {p_X} pt) =\langle \bL_*  \Cal R^{FGL} \rangle(X), \quad {\Omega}^{-i}(X \xrightarrow {p_X} pt) = \Omega_i(X).$$ 
\end{rem}

Therefore we get the following theorem
\begin{thm} The above theory $\Omega^*(X \xrightarrow {\pi_X} S)$ is an oriented Borel--Moore functor with products of geometric type on $Sch/S$ such that if $S = \op{Spec} (k) =pt$, then $\Omega^{-*}(X \to pt)$ is equal to Levine--Morel's algebraic cobordism $\Omega_*(X)$.
\end{thm}

Since $\Omega_*(X) := \Omega^{-*}(X \to pt)$ is equal to Levine--Morel's algebraic cobordism $\Omega_*(X)$, in this paper we call the above theory $\Omega^*(X \xrightarrow {\pi_X} S)$ \emph{algebraic cobordism} on $Sch/S$.

In a similar manner to the proof of the universality of $\mathcal Z^*(X \xrightarrow {\pi_X} S)$ in Theorem \ref{universality}, we can show the following universality of $\Omega^*(X \xrightarrow {\pi_X} S)$:
\begin{cor} $\Omega^*(X \xrightarrow {\pi_X} S)$ is the universal one among the oriented Borel-Moore functor with products of geometric type $A^*(X \xrightarrow {\pi_X} S)$ for $S$-schemes.
\end{cor}

\begin{rem} Motivated by the present construction, we defined a bivariant-theoretic analogue $\langle \Cal R^{Dim}\rangle(X \xrightarrow f Y) \subset \mathcal Z^*(X \xrightarrow f Y)$ and we thought that for the quotient 
$$\underline{\mathcal Z^*}(X \xrightarrow f Y):= \frac{\mathcal Z^*(X \xrightarrow f Y)}{\langle \Cal R^{Dim}\rangle(X \xrightarrow f Y) }$$
the above four operations, (i) orientation $\widetilde c_1(L)$, (ii) the bivariant product $\bullet$, (iii) the bivariant pushforward and (iv) the bivariant pullback were all well-defined. Unfortunately only biariant product was not well-defined. Thus we hope to be able to come up with a reasonable subgroup $\langle \Cal R^{Dim}\rangle(X \xrightarrow f Y) \subset \mathcal Z^*(X \xrightarrow f Y)$ such that the bivariant product on the quotient $\underline{\mathcal Z^*}(X \xrightarrow f Y)$ is well-defined.
\end{rem}

\begin{rem} In a different paper we would like to consider whether we could construct the above algebraic cobordism of $S$-schemes analogously using ``double point degeneration" of Levine-Pandharipande's construction \cite{LP}.
\end{rem}
\section{Some properties of $\Omega^i(X \xrightarrow {\op{id}_X} X)$}

We denote $\Omega^i(X \xrightarrow {\op{id}_X} X)$ by $\widetilde {\Omega}^i(X)$ to avoid confusion with Levine--Morel's algebraic cobordism $\Omega^i(X)$ in the case when $X$ is smooth. We emphasize that our $\widetilde {\Omega}^i(X)$  is defined for any scheme $X$. Similarly we denote $\Cal Z^i(X \xrightarrow {\op{id}_X} X)$ by $\widetilde {\Cal Z}^i(X)$

Here we list basic properties of $\widetilde {\Omega}^i(X)$:

\begin{enumerate}
\item For \emph{any} morphism $f:X \to Y$ we have the pullback homomorphism
$$f^*:\widetilde {\Omega}^i(Y) \to \widetilde {\Omega}^i(X).$$
It is clear that on the level of $\widetilde {\Cal Z}^i(X)$ we have the pullback homomorphism
$$f^*:\widetilde {\Cal Z}^i(Y) \to \widetilde {\Cal Z}^i(X).$$
\end{enumerate}

\begin{pro} We have the following canonical cap product:
$$\cap: \widetilde {\Omega}^i(X)\otimes \Omega_{\op{dim}X}(X) \to \Omega_{\op{dim}X -i}(X),$$
which is defined by
$$[V \xrightarrow {h} X; L_1, \cdots L_r] \cap [W \xrightarrow {k} X] := [V \times_X W \xrightarrow {h\times_X k} X; p_1^*L_1, \cdots p_1^*L_r],$$
where $p_1:V \times_X W \to V$ is the projection and note that $V \times_X W$ is smooth.
\end{pro}

In particular, when $X$ is smooth, we have the following canonical map:
$$\Cal D: \widetilde {\Omega}^i(X) \to  \Omega_{\op{dim}X -i}(X),$$
which is defined by 
$$\Cal D([V \xrightarrow h X, L_1, \cdots, L_r]):= [V \xrightarrow h X, L_1, \cdots, L_r] \cap [X \xrightarrow {\op{id}_X} X],$$
namely 
$$\Cal D([V \xrightarrow h X, L_1, \cdots, L_r]):= [V \xrightarrow h X, L_1, \cdots, L_r].$$
Since Levine and Morel define $\Omega^i(X):= \Omega_{dim X -i}(X)$ in the case when $X$ is smooth, the above canonical homomorphism
$\Cal D$ is also expressed as
$$\Cal D: \widetilde {\Omega}^i(X) \to \Omega^i(X).$$

Suppose that $X$ is not smooth. Then whenever we are given a resolution of singularities $\pi:\widetilde {X} \to X$, we have the corresponding homomorphism
$$\Cal D_{\pi}: \widetilde {\Omega}^i(X)\to \Omega_{\op{dim}X -i}(X),$$
which is defined by 
\begin{align*}
\Cal D_{\pi}([V \xrightarrow h X, L_1, \cdots, L_r]) & := [V \xrightarrow h X, L_1, \cdots, L_r] \cap [\widetilde {X} \xrightarrow {\pi} X]\\
& = [V \times_X \widetilde X \xrightarrow {h\times_X \pi } X; p_1^*L_1, \cdots p_1^*L_r],
\end{align*}
where $p_1:V \times_X \widetilde X \to V$ is the projection and note that $V \times_X \widetilde X$ is smooth. Here we note that all the resolutions of singularities make a direct system, indeed we can see this as follows. Let $\mathcal R_X$ denote the set of all the resolution of singularities of $X$. If $X$ is nonsingular, then $\mathcal R_X$ is defined to be just $\{\op{id}_X:X \to X \}$, the identity map. For two resolutions $\pi_1:\widetilde X_1 \to X$ and $\pi_2:\widetilde X_2 \to X$ we define the order $\pi_1 \leqq \pi_2$ by
$$\pi_1 \leqq \pi_2 \Longleftrightarrow \quad \exists \pi_{12}:\widetilde X_2 \to \widetilde X_1 \quad \text{such that} \quad \pi_2 = \pi_1 \circ \pi_{12}.$$
Then we have
\begin{lem} The ordered set $(\mathcal R_X, \leqq)$ is a directed set.
\end{lem} 
\begin{proof}
Indeed, for any two resolutions $\pi_1:\widetilde X_1 \to X$ and $\pi_2:\widetilde X_2 \to X$, we consider the following fiber product:
$$\CD
\widetilde X_1 \times_X \widetilde X_2  @> {\widetilde {\pi_1}} >> \widetilde X_2\\
@V {\widetilde {\pi_2}} VV @VV {\pi_2} V\\
\widetilde X_1 @> {\pi_1} >> X. \endCD 
$$
Let $\pi: \widetilde {\widetilde X_1 \times_X \widetilde X_2 } \to \widetilde X_1 \times_X \widetilde X_2$ be a resolution of singularities and let $\pi_3: \widetilde {\widetilde X_1 \times_X \widetilde X_2 } \to X$ be the composite
$$\pi_3 = \pi_1 \circ (\widetilde {\pi_2} \circ \pi) = \pi_2 \circ (\widetilde {\pi_1} \circ \pi).$$
Which means that $\pi_1 \leqq \pi_3$ and $\pi_2 \leqq \pi_3$, therefore $(\mathcal R_X, \leqq)$ is a directed set.
\end{proof}

Now, for each resolution $\pi:\widetilde X \to X  \in \mathcal R_X$, we let 
\begin{align*} 
& \Omega_{\op{dim}X -i}(X)_{\pi} \\
& := \op{Im} \mathcal D_{\pi}(\widetilde {\Omega}^i(X) ) \\
& = \left \{ [V \xrightarrow h X, L_1, \cdots, L_r] \cap [\widetilde {X} \xrightarrow {\pi} X] \, \, | \, \, [V \xrightarrow h X, L_1, \cdots, L_r]  \in \widetilde {\Omega}^i(X) \right \} \subset \Omega_{\op{dim}X -i}(X).
\end{align*}

Then we have a directed system $\left \{\Omega_{\op{dim}X -i}(X)_{\pi}, \phi_{\pi_1 \pi_2} \right \}$ where for $\pi_1 \leqq \pi_2$ the morphism $\phi_{\pi_1 \pi_2} :\Omega_{\op{dim}X -i}(X)_{\pi_1} \to \Omega_{\op{dim}X -i}(X)_{\pi_2}$ is defined by

\begin{align*}
\phi_{\pi_1 \pi_2} ( [V \xrightarrow h X, L_1, \cdots, L_r] \cap [\widetilde {X_1} \xrightarrow {\pi_1} X]) & = [V \xrightarrow h X, L_1, \cdots, L_r] \cap [\widetilde X_2 \xrightarrow {\pi_{12}} \widetilde {X_1} \xrightarrow {\pi_1} X] \\
& = [V \xrightarrow h X, L_1, \cdots, L_r] \cap [\widetilde X_2 \xrightarrow {\pi_2} X]
\end{align*}
Thus we can define the canonical map
$$\widetilde {\Cal D}:\widetilde {\Omega}^i(X) \to \varinjlim_{\pi\in \mathcal R_X} \Omega_{\op{dim}X -i}(X)_{\pi} \subset \Omega_{\op{dim}X -i}(X)$$
by
$$\widetilde {\Cal D}( [V \xrightarrow h X, L_1, \cdots, L_r]) := \varinjlim_{\pi\in \mathcal R_X} \Cal D_{\pi}( [V \xrightarrow h X, L_1, \cdots, L_r]).$$


\section{A remark on a relation with Gonzal\'ez--Karu's operational 
bivariant algebraic cobordism}

Finally we want to mention about a relation with Gonzal\'ez--Karu's operational 
bivariant algebraic cobordism \cite{GK}, which shall be denoted by $\bB_{op}^{GK}\Omega(X \to Y)$.

We expect that there is a canonical transformation

$$\tau: \Omega^*(X \xrightarrow {\pi_X} S) \to \bB_{op}^{GK}\Omega(X \xrightarrow {\pi_X} S)$$
defined as follows: for each element $[V \xrightarrow h X; L_1, \cdots, L_r]$ and for any morphism $g:S' \to S$
\begin{align*}
& \tau([V \xrightarrow h X; L_1, \cdots, L_r])\\
& := \Bigl \{h'_* \circ \widetilde c_1((g'')^*L_1) \circ \cdots \circ \widetilde c_1((g'')^*L_r) \circ ({\pi_X}' \circ h')^*: \Omega_*(S') \to \Omega_* (X') \Bigr \}_{g:S'\to S},
\end{align*}
where we consider the following fiber squares:
$$\CD
V' @> g''>> V\\
@V h' VV @VV h V\\
X' @> g' >> X \\
@V {\pi_X}' VV @VV \pi_X V\\
S' @>> g > S. \endCD
$$
We would like to treat this in a different paper.\\


\end{document}